\documentclass{amsart}

\usepackage{graphicx}
\usepackage[utf8]{inputenc}
\usepackage{nicefrac}
\usepackage{mathabx}
\usepackage{amsrefs}
\usepackage{esint}
\usepackage{mathtools}
\mathtoolsset{showonlyrefs}
\usepackage{wrapfig}

\newtheorem{theorem}{Theorem}[section]
\newtheorem{lemma}[theorem]{Lemma}

\theoremstyle{definition}
\newtheorem{definition}[theorem]{Definition}

\newtheorem{prop}[theorem]{Proposition}
\newtheorem{cor}[theorem]{Corollary}
\newtheorem{ass}{Assumption}

\newtheorem{problem}[theorem]{Open Problem}

\allowdisplaybreaks
\theoremstyle{remark}
\newtheorem{remark}[theorem]{Remark}

\numberwithin{equation}{section}

\newcommand{\E}{\mathcal{E}}

\newcommand{\dx}{\;\mathrm{d}x}
\newcommand{\dy}{\;\mathrm{d}y}
\newcommand{\dr}{\;\mathrm{d}r}
\newcommand{\dt}{\;\mathrm{d}t}
\newcommand{\ds}{\;\mathrm{d}s}
\newcommand{\dz}{\;\mathrm{d}z}
\newcommand{\HF}{\;\mathit{HYP2F1}} 

\makeatletter
\@namedef{subjclassname@2020}{%
  \textup{2020} Mathematics Subject Classification}
\makeatother

\begin{document}

\title[Elastic Flow with Obstacles]{The elastic flow with obstacles:\\ Small Obstacle Results}

\author{Marius Müller}
\address{Albert-Ludwigs-Unversität Freiburg, Mathematisches Institut, 79104 Freiburg im Breisgau}
\email{marius.mueller@math.uni-freiburg.de}
\thanks{The author was supported by the LGFG Grant (Grant no. 1705 LGFG-E) and would like to thank Anna Dall'Acqua, Francesco Nobili and Kensuke Yoshizawa for helpful discussions.}
%





\begin{abstract}
We consider a parabolic obstacle problem for Euler's elastic energy of graphs with fixed ends. We show global existence, well-posedness and subconvergence provided that the obstacle and the initial datum are suitably `small'. For symmetric cone obstacles we can improve the subconvergence to convergence. Qualitative aspects such as energy dissipation, coincidence with the obstacle and time regularity are also examined.
\end{abstract}
\maketitle

\section{Introduction}
Our object of study is the Euler-Bernoulli elastic energy 
\begin{equation}
\mathcal{W}(\gamma) := \int_\gamma \kappa^2 \; \mathrm{d}\mathbf{s},
\end{equation}
where $\gamma: I \rightarrow \mathbb{R}^2$ is a suitably smooth curve, $\kappa$ denotes its curvature and $ \mathrm{d}\mathbf{s} $ denotes the arclength parameter. If $I =(0,1)$ and $ \gamma(x) = (x,u(x))$ for some sufficiently smooth $u: I \rightarrow \mathbb{R}$ the energy becomes 
\begin{equation}
\E(u) := \int_0^1 \frac{u''(x)^2}{(1+ u'(x)^2)^\frac{5}{2}} \dx . 
\end{equation}

Since we deal with \emph{obstacle problems}, our admissible functions are required to lie above a suitable obstacle function $\psi:(0,1) \rightarrow \mathbb{R}$ which we will specify later. The boundary conditions we want to impose are `fixed ends', i.e. $u(0) = u(1) = 0$, so that the admissible set can be chosen as 
\begin{equation}\label{eq:Cpsi}
C_\psi := \{ u \in W^{2,2}(0,1) \cap W_0^{1,2}(0,1) : u \geq \psi  \; a.e. \} .
\end{equation} 
Existence (and nonexistence) of minimizers of $\E$ in $C_\psi$ has been studied in \cite{Anna} and \cite{Marius1}, minimization with slightly different frameworks has also been examined in \cite{Miura1}, \cite{Miura2},  \cite{Miura3} and \cite{Dayrens}. 

The articles  \cite{Anna}, \cite{Marius1} and \cite{Miura3} reveal that under certain smallness conditions on $\psi$ minimizers do exist whereas they do not exist in general if the obstacle is too large.

A useful necessary criterion for minimizers is the \emph{variational inequality}. More precisely -- if $u \in C_\psi$ is a minimizer, then $u$ solves
\begin{equation}\label{eq:vaaaaaarineq}
 D\E(u) (v-u) \geq 0  \quad \forall v \in C_\psi,
\end{equation}  
where $D\E$ denotes the Frechét derivative of $\E: W^{2,2}(0,1) \cap W_0^{1,2}(0,1) \rightarrow \mathbb{R}$. In the following we will also call solutions of \eqref{eq:vaaaaaarineq} \emph{constrained critical points}.


 Once minimizers are found, an object of interest is the \emph{coincidence set} $\Gamma := \{ u = \psi \}$, which forms the so-called \emph{free boundary} of the problem. For higher order variational problems like this one, a description of this free boundary is particularly challenging because of the lack of a maximum principle. 

In this article we do not want to study minimizers but rather approximation of critical points by a certain type of \emph{$L^2$-gradient flow}, called \emph{parabolic obstacle problem} in the literature.

Parabolic obstacle problems are time-dependent evolutions that flow \emph{towards solutions} of the variational inequality. Such evolutions are driven by the so-called \emph{flow variational inequality}, for short $FVI$. In our situation this reads 
\begin{equation}
(\dot{u}(t) , v- u(t))_{L^2} + D\E(u(t)) ( v- u(t)) \geq 0 \quad   \forall v \in C_\psi.
\end{equation}

Parabolic obstacle problems form a large class of time-dependent free boundary problems, sometimes also called \emph{moving boundary problems}. Here the moving boundary is given by $\Gamma_t := \{ u(t) = \psi \}$. 

 In more beneficial frameworks parabolic obstacle problems can also be seen as gradient flows in the metric space $(C_\psi, d_{L^2})$, which immediately implies that evolutions dissipate energy in a direction that is \emph{steepest possible}, cf. \cite{Ambrosio}, \cite{usersguide}.

Many authors have studied moving boundary problems driven by second order operators but recently fourth order problems have also raised a lot of interest, cf. \cite{Novaga1}, \cite{Novaga2}, \cite{Marius2}, \cite{Dayrens2}, \cite{Yoshizawa}. The energies in \cite{Novaga1}, \cite{Novaga2} are (semi-)convex which implies that the evolution can easily be regarded as a metric gradient flow in the sense of \cite{Ambrosio}, \cite{usersguide}. We emphasize that the general framework in \cite{Ambrosio}, \cite{usersguide} really relies on convexity assumptions, which $\mathcal{E}$ does not satisfy.

 In  \cite{Marius2} the lack of convexity is circumvented by looking at the gradient flow in a different flow metric, namely in the metric space $(C_\psi, d_{W^{2,2} \cap W_0^{1,2}})$. We remark that in this metric space, $\mathcal{E}$ is \emph{locally semiconvex}.
Our given energy is neither $L^2$-semiconvex nor do we want to use any other flow metric than the $L^2$-metric. For this we have to pay a price.

 Firstly, we must require that the obstacle is appropriately small to stay in a region where the elastic flow and the biharmonic heat flow show similar behavior. Most of our arguments will work by comparision to the biharmonic heat flow, controlling the nonlinearities with the various smallness requirements. 

Secondly, we are unable to fit the flow into the framework of metric gradient flows. Properties like energy dissipation are thus not immediate consequences and have to be examined seperately. Nevertheless the flow follows now dynamics that are analytically very accessible, which makes the aforementioned comparision to the biharmonic heat flow possible. This is the reason why we study this particular dynamics. 

The techniques used to construct the flow mainly rely on De Giorgi's \emph{minimizing movement scheme}, a `variational time discretization' for the problem. We remark that the evolution was constructed independently in \cite{Yoshizawa}, where the authors use the same scheme but carry out a different approach when passing to the limit. 

After the construction of our flow is finished we examine further properties such as well-posedness, size of the moving boundary, regularity and convergence behavior.  A byproduct of this study is that we show \emph{reflection symmetry} of minimizers of $\mathcal{E}$ in $C_\psi$ for some obstacles $\psi$ using \emph{symmetric decreasing rearrangements} in a setting of nonlinear higher order equations.

\section{Main Results}

In the following we discuss the basic notation and the main results. The scalar product $(\cdot, \cdot)$ will always denote the scalar product on $L^2(0,1)$. The space $W^{2,2}(0,1)\cap W_0^{1,2}(0,1)$ will always be endowed with the norm $||u||_{W^{2,2} \cap W_0^{1,2}} := ||u''||_{L^2}$, cf. \cite[Theorem 2.31]{Sweers}.

\begin{definition}[Elastic energy] 
We define the elastic energy $\mathcal{E}: W^{2,2}(0,1) \cap W_0^{1,2}(0,1)\rightarrow \mathbb{R}$ to be
\begin{equation}
    \mathcal{E}(u) := \int_0^1 \frac{u''(x)^2}{(1+ u'(x)^2)^\frac{5}{2}} \dx. 
\end{equation}
\end{definition}
\begin{remark}
With the choice of 
\begin{equation}\label{eq22}
G(s) := \int_0^s \frac{1}{(1+t^2)^\frac{5}{4}} \; \mathrm{d}t
\end{equation}
the energy becomes 
\begin{equation}
    \mathcal{E}(u) = \int_0^1 [G(u')']^2 \dx .
\end{equation}
The function $G$ is important for many quantities that we consider, hence we will fix $G$ as in \eqref{eq22} for the rest of the article.
\end{remark}

We also require some conditions on the obstacle for the entire article, which we state here. 

\begin{ass}[Assumptions on the obstacle] \label{ref:ass1}
We always assume that $\psi \in C([0,1])$ is such that $\psi(0), \psi(1) < 0 $ and there exists $x_0 \in (0,1)$ such that $\psi(x_0) >0.$ The admissible set $C_\psi$ is then defined as in \eqref{eq:Cpsi}. We define also 
\begin{equation}
I_\psi := \inf_{ u \in C_\psi} \E(u). 
\end{equation}
\end{ass}
We further introduce the constant  
\begin{equation}\label{eq:c0}
c_0 := \int_{\mathbb{R}} \frac{1}{(1+t^2)^\frac{5}{4}} \dt,
\end{equation}
which is important since \cite[Lemma 2.4]{Anna} implies that $I_\psi \leq c_0^2$ for any obstacle $\psi$ satisfying Assumption \ref{ref:ass1}. 
Another crucial observation is that $\psi \leq u_c$ for some $c \in (0,c_0)$, where
\begin{equation}
u_c(x) := \frac{2}{c \sqrt[4]{(1 + G^{-1} (\frac{c}{2}-cx)^2}} - \frac{2}{c \sqrt[4]{1 +G^{-1}(\frac{c}{2})^2}} \quad x \in (0,1),
\end{equation}
\begin{wrapfigure}{r}{6cm}
\vspace{-0.4cm}
\includegraphics[scale=0.11]{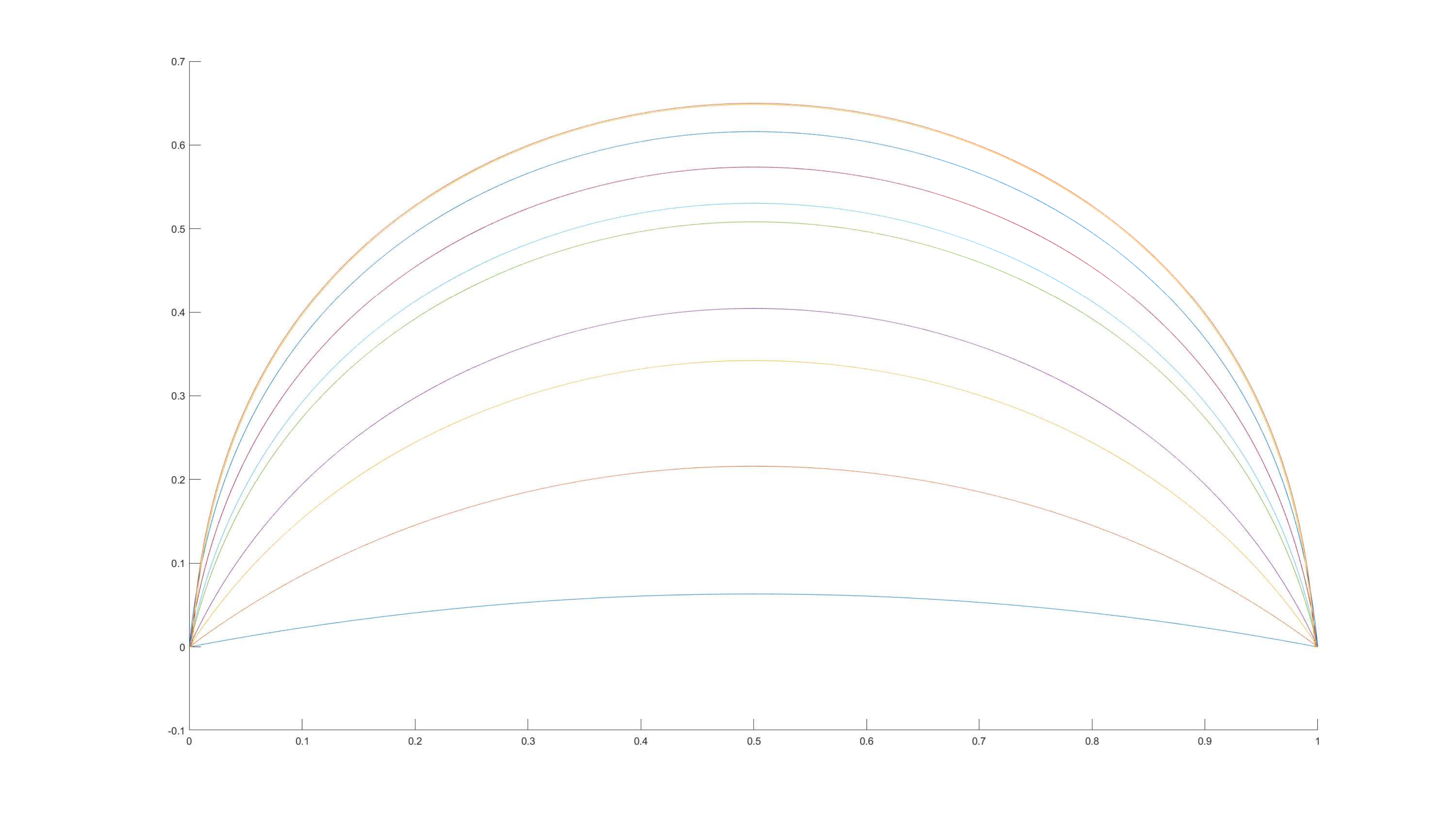}
\caption{$u_c$ for several values of $c$.}
\vspace{-0.6cm}
\end{wrapfigure} 
implies that $I_\psi \leq c^2$, cf. \cite[Lemma 2.3]{Anna}. In particular, $I_\psi$ can become arbitrarily small for small obstacles, cf. Figure 1.

In the following theorems we will always fix an initial value $u_0 \in C_\psi$  that satisfies a certain energy bound. For such an initial value to exist one usually needs a condition on $I_\psi$ which we will not write explicitly. 

\phantom{.}

\newpage

Next we define the flow we intend to construct. For this we introduce the notation $D^{1,2}((0,\infty),V):= \{ u \in W^{1,1}_{loc}((0,\infty), V) : \dot{u} \in L^2((0,\infty),V) \}$ where $\dot{u}$ denotes the weak time derivative of $u$ and $V$ is any Banach space.

\begin{definition}[$FVI$ gradient flow]\label{def:fvigrad} Let $u_0 \in C_\psi$.
We say that a function
$u \in L^\infty((0,\infty), W^{2,2}(0,1) \cap W_0^{1,2}(0,1)) \cap D^{1,2}((0,\infty); L^2(0,1))$ is an \emph{$FVI$-gradient flow} for $\mathcal{E}$ starting at $u_0$ if
\begin{itemize}
\item $u(t) \in C_\psi$ for almost every $t> 0$.
\item $t \mapsto \mathcal{E}(u(t))$ coincides almost everywhere with a nonincreasing function $\phi$ that satisfies $\phi(0) = \E(u_0)$.
\item The unique $C([0, \infty), L^2(0,1))$-representative of $u$ satisfies $u(0) = u_0$. 
\item $u$ satisfies the so-called \emph{$FVI$-inequality}
\begin{equation}\label{eq:endec}
    (\dot{u}(t) , v- u(t) ) + D\mathcal{E}(u(t))(v-u(t)) \geq 0 \quad \forall v \in C \; \; \; \; a.e. \; t > 0.
\end{equation}
\end{itemize}
\end{definition}
\begin{remark}
The existence and uniqueness of the required $C([0, \infty), L^2(0,1))$-representative follows from the Aubin-Lions lemma from which also follows that the solution lies in $C([0, \infty), C^1([0,1])).$ Whenever we address the flow from now on we will always identify it with its $C([0, \infty), C^1([0,1]))$-representative unless explicitly stated otherwise. This means in particular that evaluations at fixed times $t$ are well-defined -- at least in $C^1([0,1])$.
\end{remark}
\begin{remark}
Monotonicity of the energy does - to our knowledge - not follow from the $FVI$-equation \eqref{eq:endec}. As already mentioned it does follow in similar frameworks, cf. \cite[Proposition 2.18]{Marius2}. 
\end{remark}

Next we state our main existence theorem. 

\begin{theorem}[Existence theorem]\label{thm:existence} 
For each $u_0\in C_\psi$  such that $\mathcal{E}(u_0) < \frac{c_0^2}{4}$  there exists a (global) $FVI$ gradient flow $u$ starting at $u_0$.  
Moreover, for each FVI Gradient flow $u$ starting at $u_0$  the $C([0,\infty),C^1([0,1]) )$-representative $t\mapsto u(t)$ is bounded in $W^{2,2}(0,1)$. Furthermore $u(t) \in W^{3,\infty}(0,1)$ for almost every $t> 0 $ and for all such $t$, $u(t)$ satisfies Navier boundary conditions, i.e. $u(t)''(0) = u(t)''(1) = 0$.  
\end{theorem}

The energy threshold of $\frac{c_0^2}{4}$ is necessary for our approach since below this threshold one can obtain a control of the nonlinearities in the Euler-Lagrange equation. The same threshold (and the same control) is used in \cite{Yoshizawa}, where an existence result is obtained independently. 
If $\mathcal{E}(u_0) < \frac{c_0^2}{4}$ one can also show that the $FVI$ gradient flow starting at $u_0$ is unique, cf. \cite[Section 3]{Yoshizawa}. 

We also discuss regularity in time: As we shall see in Section 5  almost every point $t \in (0, \infty)$ is a point of continuity of $u$ in the $W^{2,2}(0,1)$-topology.


Another interesting question is whether the flow touches the obstacle in finite time. This is in particular of interest because in case that the flow does not touch the obstacle, each $FVI$-Gradient flow just coincides with a regular $L^2$ gradient flow. If this were the case, it could have been constructed with much less effort. However for small initial data there holds
\begin{prop}[Coincidence in finite time]\label{prop:touch}
For each $u_0 \in C_\psi$ such that $\mathcal{E}(u_0) < G ( \sqrt{\frac{2}{3}})^2$ there exists a sequence $t_n \rightarrow \infty$ such that $u(t_n)$ touches $\psi$. 
\end{prop}
Next, we are interested in the asymptotic behavior of the flow. For this we first examine closely what candidates for limits are available. 
\begin{definition}[Critical point] 
We say that $u \in C_\psi$ is a \emph{(constrained) critical point in $C_\psi$} if it is a solution of the \emph{variational inequality}
\begin{equation}\label{eq:symm}
 D\E(u)(v-u) \geq 0 \quad \forall v \in C_\psi . 
\end{equation}
\end{definition} 

In our classification we seek to understand only \emph{symmetric} critical points, i.e. critical points $u$ that satisfy $u= u(1- \cdot)$. The reason why those critical points are more important than the others is that under some conditions on the obstacle, the minimizers of $\E$ in $C_\psi$ can be shown to be symmetric.  The condition on the obstacle are precisely that $\psi$ itself is symmetric, ``small'' and radially decreasing i.e. $\psi$ is a decreasing function of $|x-\frac{1}{2}|$. An important special case are \emph{symmetric cone obstacles}, i.e. $\psi$ is symmetric and $\psi \vert_{[0,\frac{1}{2}]}$ is affine linear. The main technique used here is a nonlinear version of \emph{Talenti's inequality}, a classical symmetrization procedure, cf. \cite{Talenti}. Once symmetry is shown one can obtain uniqueness of critical points by the following

\begin{theorem}[Uniqueness and minimality of symmetric critical points]\label{thm:critical}
Let $\psi$ be a symmetric and radially descreasing obstacle and $u_0 \in C_\psi$ such that $\mathcal{E}(u_0) \leq G(2)^2$. Then there exists a symmetric minimizer of $\mathcal{E}$ in $C_\psi$. Moreover, if $\psi$ is a symmetric cone obstacle, this minimizer is the unique symmetric critical point in $C_\psi$.
\end{theorem}

\begin{figure}[h!]
\includegraphics[scale=0.5]{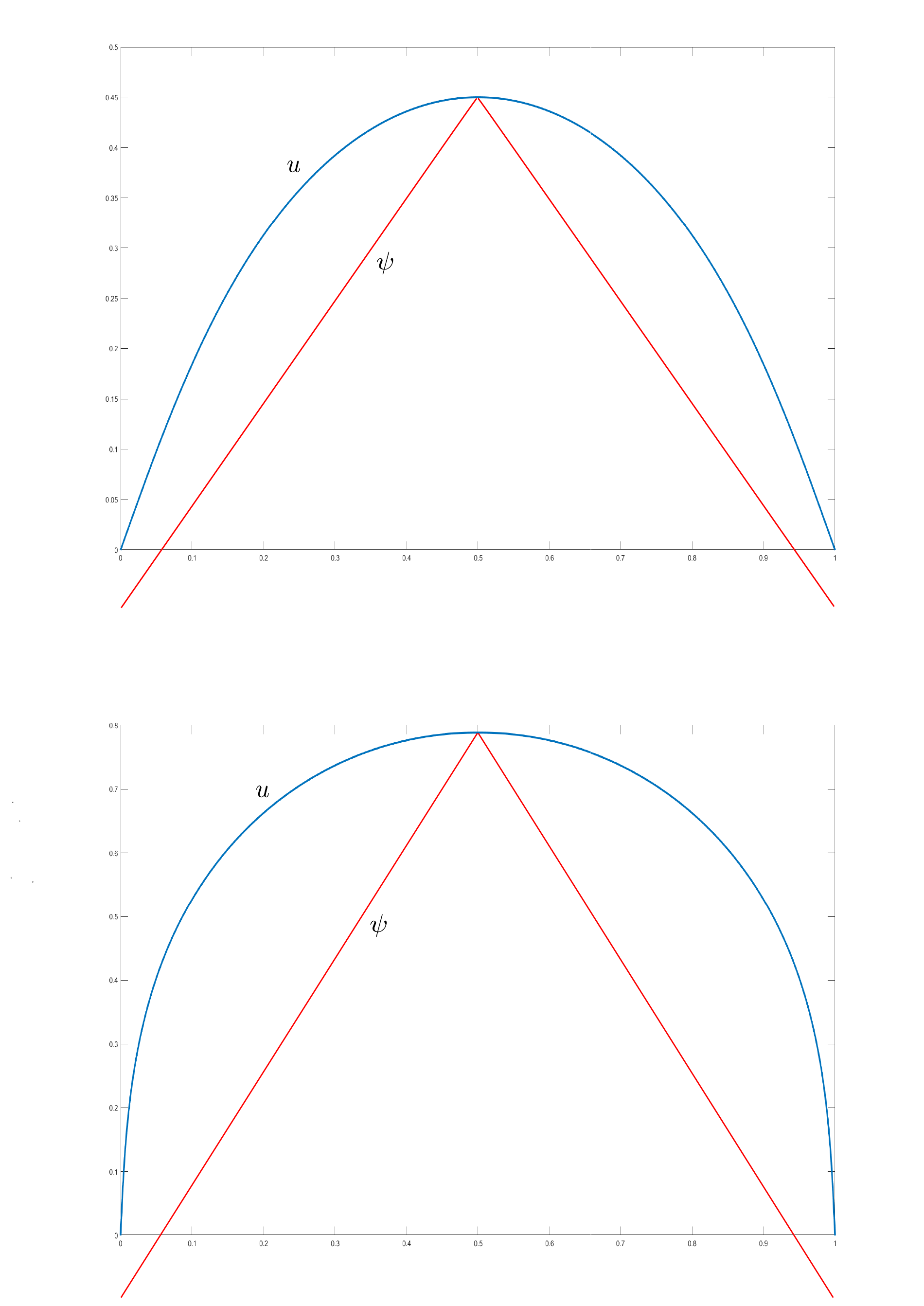}
\caption{A useful byproduct of our approach is that we can find explicit formulas for minimizers if $\psi$ is a small cone obstacle. In the situation of the first plot,  Theorem \ref{thm:critical} yields that $u$ is a minimizer and the only symmetric critical point. In the second plot, $u$ is the only symmetric critical point but the obstacle is too big to conclude with Theorem \ref{thm:critical} that $u$ is a minimizer. We strongly suspect that it is a minimizer anyway.}
\end{figure}

We remark that the symmetric critical points (and their uniqueness) have been studied independently also in \cite{Yoshizawa2} via the shooting method. Here we present a slightly different (but more or less equivalent) approach involving hypergeometric functions.    

 In Section 7 we show a subconvergence result. Here we first specify what we mean by subconvergence.
\begin{definition}[Subconvergence] \label{def:210}
Let $A \subset [0,\infty)$ be an unbounded set, $X$ be a Banach space and $u : A \rightarrow X$. Let $M \subset X$ be a set. We say that $u$ is \textrm{$X$-subconvergent} to points in $M$ if each sequence $(\theta_n) \subset  A$ such that $\theta_n \rightarrow \infty$ possesses a subsequence $\theta_{k_n}$ such that $u(\theta_{k_n})$ converges in $X$ to an element of $M$. If $A = [0, \infty)$ we say $u$ is \emph{fully} $X$-subconvergent to points in $M$. 
\end{definition}
\begin{theorem}[Subconvergence to critical points] \label{thm:convergence}
Let $u_0\in C_\psi$ be such that $\mathcal{E}(u_0) < \frac{c_0^2}{4}$. Let $u$ be an $FVI$ gradient flow starting at $u_0$. Then $u: [0,\infty) \rightarrow C^1([0,1])$ is fully $C^1([0,1])$-subconvergent to points in 
\begin{equation}\label{eq:mcrit}
  M_{crit} := \{ u_\infty \in C_\psi :   D\mathcal{E}(u_\infty) ( v- u_\infty) \geq 0  \;  \forall v \in C_\psi , \quad  \E(u_\infty) \leq \E(u_0) \}. 
\end{equation}
Moreover, for each $\epsilon >0$ there exists a set $B \subset [0,\infty)$ with $|B|< \epsilon$ such that $u_{\mid_{[0,\infty) \setminus B}}: {[0,\infty) \setminus B} \rightarrow W^{2,2}(0,1)$ is $W^{2,2}(0,1)$-subconvergent to points in $M_{crit}$.  
\end{theorem}
Here we note that smallness requirements on the obstacle are really necessary for such a subconvergence result: For large obstacles $\psi$, it is shown in \cite[Corollary 5.22]{Marius2} that no critical points exist in $C_\psi$. This shows that the energy requirement in Theorem \ref{thm:convergence} may not be omitted.

Subconvergence improves to convergence as soon as there is only one element that is still in the limit candidate set. This is the case in the situation of Theorem \ref{thm:critical}.
The following theorem summarizes many of the findings above. 
  \newpage
\begin{theorem}[Convergence for cone obstacles] \label{thm:Convconve}
Suppose that $\psi$ is a symmetric cone obstacle. Let $u_0\in C_\psi$ be symmetric and $\mathcal{E}(u_0)< \min\{ G(2)^2, \frac{c_0^2}{4}\}$.
Then there exists a unique FVI gradient flow $t \mapsto u(t)$ in $C_\psi$ that converges to the unique symmetric minimizer of $\mathcal{E}$ in $C_\psi$.

\end{theorem}
In particular we have shown that small obstacles and small initial data lead to convergent evolutions that respect the obstacle condition. 

\section{Preliminaries}

\subsection{Basic properties of the energy}
Here we discuss basic estimates and properties of $\E$ that will be useful in the following. Most energy estimates will be expressed in terms of the function $G: \mathbb{R} \rightarrow (-\frac{c_0}{2}, \frac{c_0}{2})$, where $G$ and $c_0$ are defined as in \eqref{eq22}, \eqref{eq:c0}. 
\begin{prop}[{A standard estimate for $\E$, cf. \cite[Section 4]{Anna}}]
For each $u \in C_\psi$ one has $\mathcal{E}(u) \geq G(||u'||_\infty)^2$.
\end{prop}
\begin{proof}
If $u \in C_\psi$ then $u(0) = u(1) = 0$. This implies that we can choose $\xi \in (0,1)$ such that $u'(\xi) = 0$. Also, since $u' \in C([0,1])$ we can find $\eta \in(0,1)$ such that $u'(\eta) = ||u'||_\infty$. Note that $|\xi- \eta| \leq 1$. By the Cauchy Schwarz inequality we have
\begin{align}
    \mathcal{E}(u) & = \int_0^1 (G(u')')^2 \dx \geq \left\vert \int_{\xi}^\eta (G(u')')^2 \dx \right\vert  \geq \frac{1}{|\eta - \xi|} \left( \int_{\xi}^\eta G(u')' \dx  \right)^2\\ &  = \frac{1}{|\eta - \xi|} (G(u'(\eta)) - G(u'(\xi)))^2 \geq G(||u'||_\infty)^2. \qedhere
\end{align}
\end{proof}
\begin{remark}\label{ref:remeng}
Note that each $v \in C_\psi$ such that $\mathcal{E}(v) < \frac{c_0^2}{4} $ must then satisfy $||v'||_\infty < G^{-1} \left( \sqrt{ \E(v) } \right)< \infty $. Using this and \cite[Lemma 2.5]{Anna} we find that $\inf_{ u \in C_\psi} \E(u) < \frac{c_0^2}{4}$ implies existence of a global minimizer of $\E$ in $C_\psi$. It is also worth noting that each $v \in C_\psi$ with $\E(v) < \frac{c_0^2}{4}$ satisfies 
\begin{equation}\label{eq:standardW22}
||v||_{W^{2,2}\cap W_0^{1,2}}^2 \leq \E(v) ( 1+ G^{-1}(\sqrt{\E(v)})^2)^\frac{5}{2} .
\end{equation}
These observations are the crucial reason for the energy bounds in the statement of Theorem \ref{thm:existence} and Theorem \ref{thm:convergence}.
\end{remark}
In the following proposition we discuss first properties of the Frechét derivative $D\E$. Most of those computations have already been made in \cite{Anna},   \cite{Marius1}.
\begin{prop}[Explicit formulas for the derivative]\label{prop:33}
For each $u,\phi \in W^{2,2}(0,1) \cap W_0^{1,2}(0,1) $ one has 
\begin{equation}\label{eq:eulernaiv}
    D\mathcal{E}(u)(\phi) = 2 \int_0^1 \frac{u'' \phi''}{(1+ u'^2)^\frac{5}{2}} \dx - 5 \int_0^1 \frac{u''^2 u' \phi'}{(1+ u'^2)^\frac{7}{2}} \dx. 
\end{equation}
If additionally $u \in W^{3,1}(0,1)$ and $u''(0) = u''(1) = 0$ then 
\begin{equation}\label{eq:euler}
D\mathcal{E}(u)(\phi) = - 2 \int_0^1 \frac{A_u'\phi'}{(1+(u')^2)^\frac{5}{4}} \dx,
\end{equation}
where 
\begin{equation}
A_u := \frac{u''}{(1+u'^2)^\frac{5}{4}}.
\end{equation}
\end{prop}
\begin{proof}
Equation \eqref{eq:eulernaiv} can be found in \cite[Equation 1.5]{Anna}. If $u$ is now as in the second part of the claim we can perform an integration by parts in the first summand to get 
\begin{align}
D\E(u)(\phi) & = \left[ \frac{u''\phi'}{(1+ u'^2)^\frac{5}{2}} \right]_0^1 - 2 \int_0^1 \frac{u''' \phi'}{(1+ u'^2)^\frac{5}{2}} \dx + 5  \int_0^1 \frac{u''^2 u' \phi'}{(1+ u'^2)^\frac{7}{2}} \dx
\end{align}
Since the boundary terms vanish by assumption we only end up with the last two integrals, whereupon \eqref{eq:euler} can easily be verified using the product rule. 
\end{proof}
\begin{remark}
The astoundingly compact formula \eqref{eq:euler} was already known to Euler, see \cite{Euler}, and will be of great use for us. The notation $A_u := \frac{u''}{(1+ u'^2)^\frac{5}{4}}$ will also be used throughout the article. 
\end{remark}
\begin{remark}\label{rem:opnormemb}
Another useful consequence of \eqref{eq:eulernaiv} is that for all $u, \phi \in W^{2,2}(0,1) \cap W_0^{1,2}(0,1)$ one has 
\begin{equation}\label{eq:naivgrad}
|D\E(u)(\phi)| \leq  ( 2 ||u||_{W^{2,2}} + 5 ||u||_{W^{2,2}}^2 ) ||\phi||_{W^{2,2}}. 
\end{equation}
Here we have used that $||\phi'||_{L^\infty} \leq || \phi''||_{L^2}$ for all $\phi \in W^{2,2} \cap W_0^{1,2}$. 
\end{remark}

\subsection{Basic properties of FVI gradient flows}
In this section we will briefly discuss why the $FVI$ gradient flow generalizes the concept of an $L^2$- gradient flow. Furthermore we will discuss some basic regularity properties that follow immediately by the definition. Recall that for a Hilbert space $H$ that is dense in $L^2$, a Frechét differentiable functional $\mathcal{F}: H \rightarrow \mathbb{R}$ is said to have an \emph{$L^2$-gradient} at $u \in H$ if $D\mathcal{F}(u)  \in H^*$ possesses an extension to a linear continuous functional in $(L^2)^*$. We denote by $\nabla_{L^2} \mathcal{F}(u) \in L^2$ the representing element of this functional in $L^2$.   
\begin{prop}[Consistency with $L^2$-gradient flows]\label{prop:concistncy}
Let $u$ be an $FVI$ gradient flow for $\mathcal{E}$ in $C_\psi$. Let $t > 0 $ be a point where $\{ u(t) = \psi \} = \emptyset$ and the $FVI$ (i.e. \eqref{eq:endec}) holds. Then $\E$ posesses an $L^2$-gradient at $u(t)$ and one has 
\begin{equation}
\dot{u}(t) = - \nabla_{L^2} \E(u(t)). 
\end{equation}
\end{prop}
\begin{proof}
If $\{ u(t) = \psi \} = \emptyset$ then $u(t) > \psi$ on $[0,1]$. As $u(t), \psi  \in C([0,1])$ there exists $\delta> 0 $ such that $u(t) - \psi > \delta$ on $[0,1]$. Now let $\phi \in W^{2,2} \cap W_0^{1,2}$. Then by Sobolev embedding $\phi \in C([0,1])$. By the choice of $\delta$ one has $v := u(t) + \epsilon \phi \in C_\psi$ for all $\epsilon \in \mathbb{R}  : |\epsilon| < \frac{1}{||\phi||_\infty}\delta$. Testing \eqref{eq:endec} with the function $v$ we have just chosen we find that 
\begin{equation}
\epsilon \left( ( \dot{u}(t) , \phi) + D\E(u(t))( \phi) \right)  \geq 0 .
\end{equation}
Looking first at a positive and then at a negative value of $\epsilon$ and dividing both times by $\epsilon$ we find 
\begin{equation}
 D\E(u(t))( \phi) = ( -\dot{u}(t),\phi) \quad \forall \phi \in W^{2,2}(0,1) \cap W_0^{1,2}(0,1)  .
\end{equation}
Since $W^{2,2} \cap W_0^{1,2}$ is dense in $L^2$ we find that $D\E(u(t))$ can be extended to an element of $L^2(0,1)^*$, represented by $-\dot{u}$. By the very definition of $L^2$-gradient follows that 
$
-\dot{u}(t) =  \nabla_{L^2} \E(u(t))
$
and hence the claim. 
\end{proof}
The $FVI$ gradient flow starts with significantly less regularity in time than the one in \cite{Marius2}. However we can extract some immediate regularity properties: Here we expose a basic feature that will be very important for our analysis: Uniform  $C^{0, \frac{1}{2}} ((0, \infty); L^2(0,1))$-estimates.

\begin{prop}[Uniform Hölder continuity in time] \label{prop:holder}
Let $u_0 \in C_\psi$ and $u$ be an $FVI$ gradient flow starting at $u_0$. Then there exists $D> 0$ such that for all $t,s > 0 $ one has 
\begin{equation}
||u(t)- u(s)||_{L^2(0,1)} \leq D \sqrt{|t-s|}.
\end{equation}
\end{prop}
\begin{proof}
Let $u_0 ,u,s,t $ be as in the statement. 
\begin{equation}
\left\Vert u(t) - u(s) \right\Vert_{L^2(0,1)} \leq \left\Vert \int_s^t \dot{u}(r) \dr \right\Vert_{L^2(0,1)} \leq \sqrt{|t-s|} \;  ||\dot{u}||_{L^2((0,\infty), L^2(0,1))}
\end{equation}
Choosing $D:= ||\dot{u}||_{L^2((0,\infty), L^2(0,1))}$ which is finite by Definition \ref{def:fvigrad}, we obtain the claim.
\end{proof}

\section{Existence Theory}\label{sec:exi}
In this section we construct the $FVI$ gradient flow by a variational discretization scheme. We first show existence of so-called \emph{discrete flow trajectories}, which we will define. The discrete stepwidth will always be denoted by $\tau > 0$. Once the discrete trajectories are defined we can look at their asymptotics as $\tau \rightarrow 0$. To get desirable limit trajectories we need a suitable compactness, which we will achieve by discussing additional regularity properties of the discrete trajectories in parabolic function spaces.
\subsection{Construction} 
\begin{lemma}[Discretization scheme, proof in Appendix \ref{app:A}] \label{lem:discr}
Suppose that $f \in C_\psi$ is such that  $\mathcal{E}(f) < \frac{c_0^2}{4}$. Then for each $\tau > 0 $ the energy $\Phi_\tau^f: C_\psi \rightarrow \mathbb{R}$ defined by 
\begin{equation}
\Phi_\tau^f (u) := \E(u) + \frac{1}{2\tau} || u-f||_{L^2}^2 
\end{equation}
has a minimizer in $C_\psi$. Any such minimizer $w \in C_\psi$ satisfies
\begin{equation}\label{eq:verreg}
\frac{1}{\tau} (w-f , v-w ) + D\E(w) (v-w) \geq 0 \quad \forall v \in C_\psi.
\end{equation}
\end{lemma}

\begin{definition} (Minimizing movements) 
Let $u_0 \in C_\psi$ be such that $\E(u_0) < \frac{c_0^2}{4}$ and $\tau > 0$. We define iteratively $u_{0,\tau}:= u_{0\tau} := u_0$ and choose for each $k \in \mathbb{N}$
\begin{equation}\label{eq:miniprob}
    u_{(k+1)\tau} \in \arg \min_{ u \in C} \left(  \mathcal{E}(u) + \frac{1}{2\tau} ||u-u_{k\tau}||_{L^2}^2 \right) = \arg \min_{ u \in C} \Phi_{\tau}^{u_{k\tau}} (u) .
\end{equation}
We also define the \emph{piecewise constant interpolation} $\overline{u}^\tau : [0, \infty) \rightarrow C_\psi$ by $\overline{u}^{\tau}(0) = u_0$ and
\begin{equation}
    \overline{u}^\tau (t) := u_{(k+1) \tau}, \quad t \in (k\tau, (k+1) \tau]
\end{equation}
as well as the \emph{piecewise linear interpolation} $u^\tau: (0, \infty) \rightarrow C_\psi$ by
\begin{equation}
    u^\tau(t):= u_{k\tau} + \frac{t- k \tau}{\tau} ( u_{(k+1)\tau} - u_{k \tau } ) \quad t \in (k \tau , (k+1) \tau].
\end{equation}
\end{definition}
\begin{remark}\label{rem:minibound}
That the minimum problem in \eqref{eq:miniprob} has a solution for all $\tau >0$ and $k \in  \mathbb{N}$ is due to Lemma \ref{lem:discr}. To ensure that the Lemma is applicable for all $k \in \mathbb{N}$ we have to check that $\E(u_{k \tau } )< \frac{c_0^2}{4}$ for all $k \in \mathbb{N}$ and $\tau > 0$. This follows by induction since for all $k \in \mathbb{N}$  it holds that $\mathcal{E}(u_{(k+1) \tau }) \leq \mathcal{E}(u_{k\tau})$. The latter inequality is an immediate consequence of the observation that $\Phi_\tau^{u_{k\tau}}(u_{(k+1) \tau}) \leq \Phi_\tau^{u_{k\tau}} ( u_{k \tau} )$ by \eqref{eq:miniprob}. Another noteworthy implication of this inequality is that for all $\tau > 0$ the map $[0, \infty) \ni t \mapsto \E(\overline{u}^\tau(t)) \in \mathbb{R}$ is nonincreasing and coincides with $\E(u_0)$ at $t = 0$. Monotonicity of the energy is not immediate for the piecewise linear interpolations, which reveals an important advantage of the pievewise constant interpolation. 
\end{remark}
\begin{remark}
Another consequence of the inequality $\Phi_\tau^{u_{k\tau}}(u_{(k+1) \tau}) \leq \Phi_\tau^{u_{k\tau}} ( u_{k \tau} )$ that we will use frequently is that 
\begin{equation}\label{eq:stanminimov} 
\frac{1}{2\tau} || u_{(k+1) \tau } - u_{k \tau}||_{L^2}^2 \leq \E(u_{(k+1) \tau}) - \E(u_{k\tau}) .
\end{equation}
\end{remark}
\begin{remark}\label{rem:interpoli}
Note that the piecewise linear interpolation is weakly differentiable in $t$ and satisfies 
\begin{align}
    \int_0^\infty || \dot{u}^\tau(t) ||_{L^2}^2 \dt & = \sum_{k = 0}^\infty  \int_{k\tau}^{(k + 1) \tau}  \frac{||u_{(k+1) \tau}- u_{k \tau}||_{L_2}^2 }{\tau^2} \dt = \sum_{k = 0}^\infty \frac{1}{\tau} ||u_{(k + 1)\tau } - u_{k \tau }||^2 \\ &  \leq \sum_{k = 0}^\infty 2 ( \mathcal{E}(u_{(k+1) \tau})  - \E(u_{k\tau}) )  \leq 2\mathcal{E}(u_0).
\end{align}
Hence we have a uniform estimate for $||\dot{u}^\tau||_{L^2((0, \infty), L^2(0,1)) }$ independently of $\tau$. Moreover, for all $T>0$ $u^\tau$ lies in $W^{1,2}((0,T),L^2(0,1))$ with norm bounded by 
\begin{equation}\label{eq:l2buutadot}
||u^\tau||_{W^{1,2}((0,T),L^2(0,1))}^2 \leq C_T ( ||u_0||_{L^2}^2 +  2 \E(u_0)),
\end{equation}
where $C_T$ is a constant that depends only on $T$ and not on $\tau$. 
\end{remark}

\begin{remark}
For the minimization problems in \eqref{eq:miniprob}, \eqref{eq:verreg} yields a variational inequality that reads 
\begin{equation}\label{eq:VARDISCR}
    (\dot{u}^\tau(t), v- \overline{u}^\tau(t) ) + D\mathcal{E}(\overline{u}^\tau(t))( v- \overline{u}^\tau(t) ) \geq  0 \quad \forall t > 0 \quad \forall v \in C_\psi. 
\end{equation}
Notice that both $u^\tau$ and $\bar{u}^{\tau}$ play a role in this variational inequality. 
\end{remark}
\begin{remark}\label{rem:diffinterpol}
Another fact we will make use of is that the $L^2$ distance of both defined interpolations can be uniformly controlled in time, more precisely  
\begin{equation}
 ||u^\tau(t) - \overline{u}^\tau(t)||_{L^2} \leq \sqrt{2\tau} \sqrt{\mathcal{E}(u_0)} \quad \forall t >0 .
\end{equation}
This is an immediate consequence of \eqref{eq:stanminimov}. 
\end{remark} 
\begin{lemma}[Uniform $W^{2,2}$-bounds, proof in Appendix \ref{app:A}] \label{lem:miii}
Let $u_0 \in C_\psi$ such that $\mathcal{E}(u_0) < \frac{c_0^2}{4}$. Then $(u^\tau)_{\tau > 0}$ is bounded in $L^\infty((0,\infty), W^{2,2}(0,1)) $
\end{lemma}
With the next lemma, we can obtain a global limit trajectory of $u^{\tau_n}$ for a carefully chosen sequence $\tau_n \rightarrow 0$. The convergence will unfortunately not be good enough to show that this limit trajectory is an $FVI$ gradient flow. To this end we have to obtain more regularity first and work with both $\bar{u}^{\tau_n}$ and $u^{\tau_n}$. 

 Here we fix the right subsequence to consider for the additional regularity. 
 
 A property that we will use very often in the arguments to come is that weak topologies have the \emph{Urysohn property}, i.e. a sequence converges weakly if and only if each subsequence has a weakly convergent subsequence and the limit of all those subsequences coincide. 
 
 Another main tool will be the classical Aubin-Lions lemma or more modern versions thereof. 
\begin{lemma}[The limit trajectory]\label{lemma:limitcase}
Let $u_0 \in C_\psi$ be such that $\mathcal{E}(u_0) < \frac{c_0^2}{4}$. Then there exists a sequence $\tau_n \rightarrow 0$  and $u \in L^\infty((0,\infty),W^{2,2}(0,1)) \cap D^{1,2}((0,\infty),L^2(0,1)) $ such that for all $T > 0$, $(u^{\tau_n})_{n = 1}^\infty$ converges to $u$ weakly in $W^{1,2}((0,T),L^2(0,1))$ and strongly in $C([0,T],C^1([0,1]))$. Moreover, for each $t > 0$ $u^{\tau_n}(t) \rightharpoonup u(t)$ in $W^{2,2}(0,1)$. In particular $u(t) \in C_\psi$ for each $t> 0$ (and not just for Lebesgue a.e. $t >0$). 
\end{lemma}
\begin{proof}
    By the Aubin-Lions lemma  $L^\infty((0,2), W^{2,2}(0,1))\cap W^{1,2}((0,2), L^2(0,1)) $ embeds compactly into $C([0,2], C^1([0,1]))$ and continuously into $W^{1,2}((0,2), L^2(0,1))$. Hence one can find a sequence $\tau_n^2 \rightarrow 0$ such that
     $u^{\tau_n^2}$ converges  to  some $u_2$ strongly in $C([0,2], C^1([0,1])) $ and weakly in $W^{1,2}((0,2), L^2(0,1))$. Now, again by the Aubin-Lions lemma $L^\infty((0,3) , W^{2,2}(0,1)) \cap W^{1,2}((0,3), L^2(0,1)) $ embeds compactly into $C([0,3], C^1([0,1]))$ and continuously into $W^{1,2}((0,3), L^2(0,1))$ and therefore we can find a subsequence $(\tau_n^3)_{n = 1}^\infty$ of $(\tau_n^2)_{n = 1}^\infty$ such that $u^{\tau_n^3}$ converges
      to some $u_3$ strongly in $C([0,3], C^1([0,1])) $ and weakly in $W^{1,2}((0,3), L^2(0,1))$. Since uniform convergence implies pointwise convergence, we find that $u_3 =u_2 $ on $[0,2]$. Iteratively we can construct  nested subsequences $(\tau_n^l)_{n = 1}^\infty \subset (\tau_n^{l+1})_{n = 1}^\infty$ and $u \in C([0,\infty),C^1([0,1]))$ such that $u^{\tau_n^l}$ converges to $u_{\mid_{[0,l]}}$ strongly in $C([0,l], C^1([0,1]))$ and weakly in $W^{1,2}((0,l), L^2(0,1))$. Taking now
       the subsequence $(\tau_n^n)_{n = 1}^\infty$ we find that $u^{\tau_n^n}$ converges to $u$ strongly in $C([0,T], C^1([0,1]))$ and weakly in $W^{1,2}((0,T),L^2(0,1))$ for each $T> 0$.
       
        It remains to prove that $u$ lies in $L^\infty((0,\infty),W^{2,2}(0,1))\cap D^{1,2}((0,\infty),L^2(0,1))$ and $u^\tau(t)$ converges weakly to $u(t)$ in $W^{2,2}(0,1)$ for each $t > 0$. We start with the latter assertion. We know that for each $t> 0$, $(u^{\tau_n^n}(t))$ converges to $u(t)$ in $C^1([0,1])$. Since also $||u^{\tau_n^n}(t)||_{W^{2,2}}$ is uniformly 
       bounded by Lemma \ref{lem:miii}, we obtain that each subsequence of $u^{\tau_n^n}$ has a weakly $W^{2,2}$-convergent subsequence to some $\widetilde{u} \in W^{2,2}(0,1)$ that may depend on the choice of the subsequence. However, by uniqueness of limits in $C^1([0,1])$, we get that $\widetilde{u} = u(t)$ for each possible choice of a subsequence. This being
        shown, the Urysohn property of
         weak convergence implies that $u^{\tau_n^n}(t) \rightharpoonup u(t)$ in $W^{2,2}(0,1).$ Note that $u(t) \in C_\psi$ as $C_\psi$ is weakly closed in $W^{2,2}(0,1)$. The uniform boundedness of $(u^{\tau_n^n})_{n = 1}^\infty$ in $L^\infty( (0,\infty) ,W^{2,2}(0,1))$ implies then that $u \in L^\infty((0,\infty),W^{2,2}(0,1))$. We now show that $u \in D^{1,2}((0,\infty),L^2(0,1))$. For weak differentiability of $u$ on $(0, \infty)$ fix $\phi \in C_0^\infty((0, \infty);\mathbb{R})$. Observe that there exists $T> 0 $ such that $\mathrm{supp}(\phi) \subset (0,T)$. Since $u^{\tau_n^n} \rightharpoonup u$  in $W^{1,2}((0,T),L^2(0,1)) $ we obtain 
         \begin{equation}
         \int_0^\infty u \dot{\phi} \dt = \mathrm{wlim}_{n \rightarrow \infty} \int_0^T u^{\tau_n^n} \dot{\phi} \dt = -  \mathrm{wlim}_{n \rightarrow \infty} \int_0^T \dot{u}^{\tau_n^n} \phi \dt  = - \int_0^\infty \dot{u} \phi \dt ,
         \end{equation}
         where $\mathrm{wlim}$ denotes the weak limit in $L^2(0,1)$.
It only remains to show that $\dot{u} \in L^2((0,\infty),L^2(0,1))$. To this end, note that for each $N \in \mathbb{N}$ one has by  Remark \ref{rem:minibound}
    \begin{equation}
        \int_0^N ||\dot{u}(t)||_{L^2}^2 \dt \leq \liminf_{n \rightarrow \infty} \int_0^N || \dot{u}^{\tau_n^n}(t) ||_{L^2}^2 \dt \leq 2 \mathcal{E}(u_0) ,  
    \end{equation}
    which is a bound that is independent of $N$. Letting $N \rightarrow \infty$ the monotone convergence theorem implies the claim. 
  \end{proof}
  In the next lemma we show an estimate that will later account for $L^2((0, T), W^{3,\infty}(0,1))$-bound of $(u^\tau)_{\tau > 0}$. This will be the needed extra regularity to pass to the limit in the energy space. Another useful byproduct are the Navier boundary conditions that are a natural consequence of the underlying variational inequalities (cf. \eqref{eq:VARDISCR}).  
  \begin{lemma}[$W^{3,\infty}$-bounds and Navier boundary conditions, proof in Appendix \ref{app:A}]\label{lem:w3infbds}
  Let $u_0 \in C_\psi$ be such that $\mathcal{E}(u_0) < \frac{c_0^2}{4}$. Then there exist $C,D > 0$ such that for each $\tau > 0$ and each $t> 0$, $\overline{u}^{\tau}(t) \in W^{3,\infty}(0 ,1)$, $(u^{\tau}(t))''(0) = (u^{\tau}(t))''(1) = 0$. Moreover
  \begin{equation}\label{eq:w3inf}
  ||\overline{u}^\tau(t)'''||_{L^\infty} \leq C + D || \dot{u}^\tau (t) ||_{L^2(0,1)} \quad \forall t >0  
  \end{equation}
  for constants $C,D$ that may depend on $u_0$ and the obstacle but not on $\tau$. 
  \end{lemma}
 
  \begin{lemma}[Almost-everywhere convergence in energy space]\label{lem:a.e}
  Let $u_0, u$ be as in Lemma \ref{lemma:limitcase}. Then, for each $T> 0 $,  $(u^\tau)_{\tau> 0}$ is uniformly bounded in $L^2((0,T), W^{3,\infty}) \cap W^{1,2}((0,T),L^2)$. Moreover there exists a sequence $\tau_n \rightarrow 0 $ such that for each $T> 0 $, $(u^{\tau_n})_{n = 1}^\infty$ converges to $u$ weakly in $L^2((0,T), W^{3,2})$, strongly in $L^2((0,T),C^2([0,1]))$ and pointwise almost everywhere in $C^2([0,1])$. In particular, $u(t) \in W^{3,2}(0,1)$ and $u(t)''(0) = u(t)''(1) = 0$ for almost every $t$. 
  \end{lemma}
  \begin{proof}
  Let $T>0$ be fixed.
  That $(u^\tau)_{\tau > 0}$ is uniformly bounded in  $W^{1,2}((0,T),L^2)$ has already been shown in Remark \ref{rem:interpoli}. 
  Fix $t, \tau > 0$ and fix $k \in \mathbb{N}_0$ such that $t \in (k \tau, (k+1)\tau].$ 
  By \eqref{eq:w3inf} one has 
  \begin{align*}
  ||u^\tau(t)||_{W^{3,\infty}}  & \leq ||u_{k \tau} ||_{W^{3,\infty}} + ||u_{(k+1)\tau}||_{W^{3,\infty}}  = ||\overline{u}^\tau(k\tau)||_{W^{3,\infty}} + ||\overline{u}^\tau((k+1)\tau)||_{W^{3,\infty}} \\ &  \leq ||u^\tau||_{L^\infty((0,T),W^{2,2})} + 2C + D( || \dot{u}^\tau(k\tau)||_{L^2} + || \dot{u}^\tau( (k+1) \tau) ||_{L^2} ) ,
  \end{align*}
  where $C,D$ are chosen as in \eqref{eq:w3inf}. 
  For the next computation we set for convenience of notation $u_{-1\tau} := u_0$. We can use the above estimate and \eqref{eq:stanminimov} to find   
  \begin{align*}
  \int_0^T ||u^\tau(t)||_{W^{3,\infty}}^2 \; \mathrm{d}t & \leq 2 T ( ||u^\tau||_{L^\infty((0,T),W^{2,2}} + 2C ) ^2  \\ & \quad + 4D^2 \left( \sum_{k = 0 }^\infty \tau (  || \dot{u}^\tau(k\tau) ||_{L^2}^2 + || \dot{u}^\tau ((k+1) \tau )||_{L^2}^2) \right)\\
  & \leq 2 T ( ||u^\tau||_{L^\infty((0,T),W^{2,2}} + 2C ) ^2  \\ & \quad + 4D^2 \left( \sum_{k = 0 }^\infty  \frac{1}{\tau}( ||u_{(k + 1)\tau}- u_{k \tau}  ||^2 + || u_{k\tau} - u_{(k-1)\tau }||^2) \right)
  \\ & = 2 T ( ||u^\tau||_{L^\infty((0,T),W^{2,2})} + 2C ) ^2 \\ & \quad + 4D^2 \left( \sum_{k = 0 }^\infty 2( \mathcal{E}(u_{(k- 1)\tau})- \mathcal{E}(u_{(k+1) \tau }))  \right)  \\ & \leq 2T( ||u^\tau||_{L^\infty((0,T),W^{2,2}}) + 2C ) ^2 + 16D^2 \mathcal{E}(u_0). 
  \end{align*}
  We infer that $(u^\tau)_{\tau >0 }$ is bounded in $L^2((0,T), W^{3,\infty}) \cap W^{1,2}((0,T),L^2)$, which embeds by the Aubin-Lions-Lemma compactly into $L^2((0,T), C^2([0,1])$. Let $\tau_n$ be the sequence constructed in Lemma \ref{lemma:limitcase}. Then by the $L^2((0,T), W^{3,\infty}) \cap W^{1,2}((0,T),L^2)$-bound each subsequence of $\tau_n$ must have another subsequence along which $(u^{\tau_n})_{n  = 1}^\infty$ converges weakly in $L^2((0,T), W^{3,2}(0,1))$ and strongly in $L^2((0,T),C^2([0,1]))$. Because of uniqueness of weak  limits in $L^2((0,T), W^{2,2}) $ we deduce that all those subsequences must converge to the same $u$ as constructed in Lemma \ref{lemma:limitcase}. By the Urysohn property $(u^{\tau_n})_{n = 1}^\infty$ converges to $u$ strongly in $L^2((0,T),C^2([0,1]))$ and weakly in $L^2((0,T), W^{3,2}(0,1) )$. Convergence in the claimed spaces follows as $T> 0$ was arbitrary.  Choosing a further subsequence of $(\tau_n)_{n = 1}^\infty$ we may also assume that $u^{\tau_n} \rightarrow u$  pointwise almost everywhere in $C^2([0,1])$ as $L^2$-convergence of Banach-space valued functions implies the existence of a pointwise almost everywhere convergent subsequence. That $u(t)''(0)= u(t)''(1) = 0$ for almost every $t>0$ is then an immediate consequence of this fact. 
  \end{proof}
  
  So far we have shown convergence of the piecewise linear interpolations. As mentioned in Remark \ref{rem:minibound} we also need results on the behavior of the piecewise constant interpolations to control the energy. 
  \begin{lemma}[Precompactness of piecewise constant interpolation]\label{lem:discraubinlions}
  Let $u_0$ be as before. Then $(\overline{u}^\tau)_{\tau \in (0,1)} $ is precompact in $L^2((0,T),C^2([0,1]))$ for each $T>0$. 
  \end{lemma}
  \begin{proof}
  The proof relies on a discrete version of the Aubin-Lions lemma -- the so-called discrete Aubin-Lions-Dubinskii lemma, see \cite[Theorem 1]{Dreher}. To apply this we just need to show that for all $T>0$ the expression 
  \begin{equation}
  \frac{1}{\tau}||\overline{u}^\tau(\cdot-\tau) - \overline{u}^\tau(\cdot) ||_{L^1((0,T),L^2(0,1))} + ||\overline{u}^\tau||_{L^2((0,T),W^{3,\infty}(0,1))} 
\end{equation}   
is uniformly bounded in $\tau$. The claim follows then since the embedding $W^{3,\infty}(0,1) \hookrightarrow C^2([0,1])$ is compact. That the second summand is uniformly bounded in $\tau$ follow from \eqref{eq:w3inf} and \eqref{eq:l2buutadot}.
 For the first summand let $N_\tau \in \mathbb{N}$ be such that $(N_\tau-1)\tau \leq T \leq N_\tau \tau$ and calculate using \eqref{eq:stanminimov}
\begin{align}
 \frac{1}{\tau}||\overline{u}^\tau(\cdot-\tau)& - \overline{u}^\tau(\cdot) ||_{L^1((0,T),L^2(0,1))}  \leq \frac{1}{\tau} \sum_{k = 0}^{N_\tau} \tau ||u_{(k+1)\tau}- u_{k\tau}||_{L^2} \\ & = \sum_{k = 0}^{N_\tau}  ||u_{(k+1)\tau}- u_{k\tau}||_{L^2}  \leq \sqrt{2\tau } \sum_{k = 0 }^{N_\tau} \sqrt{ \mathcal{E}(u_{k \tau }) - \mathcal{E}(u_{(k-1)\tau })}
 \\ &  = \sqrt{2\tau} \sqrt{N_\tau} \left( \sum_{k = 0}^{N_\tau} \mathcal{E}(u_{k \tau }) - \mathcal{E}(u_{(k-1)\tau } ) \right)^\frac{1}{2} 
  = \sqrt{2(T+ 1)} \left( \E(u_0) \right)^\frac{1}{2}.
\end{align}
Hence \cite[Theorem 1]{Dreher} is applicable and the claim follows. 
  \end{proof}
  \begin{cor}[$FVI$ gradient flow property]\label{cor:existiert}
  Let $u_0, u $ be as in Lemma \ref{lemma:limitcase}. Then $u$ is a $FVI$-Gradient Flow.
  \end{cor}
  \begin{proof}
  The fact that $u(t) \in C_\psi$ for almost every $t > 0 $ follows from the fact that $C_\psi$ is weakly closed in $W^{2,2}(0,1)$.  For the proof of the $FVI$ inequality we choose $v \in C_\psi$.
  Let $(u^{\tau_n})_{n = 1}^\infty$ be a sequence chosen as in Lemma \ref{lem:a.e}. Further let $0<a<b$ be arbitrary.  Integrating \eqref{eq:VARDISCR} we find that 
\begin{equation}
\int_a^b ( \dot{u}^{\tau_n}(t) , v- \overline{u}^{\tau_n}(t) ) \dt + \int_a^b D\mathcal{E}(\overline{u}^{\tau_n}(t)) (v- \overline{u}^{\tau_n}(t)) \dt  \geq 0 .
\end{equation}  
  Since $u^{\tau_n}$ converges to $u$ uniformly in $L^2(0,1)$ we can infer from Remark \ref{rem:diffinterpol} that for each $t > 0$ $\overline{u}^{\tau_n}(t)$ converges to $u(t)$ in $L^2(0,1)$. We also infer from Lemma \ref{lem:discraubinlions} that -- after choosing an approprate subsequence of $(\tau_n)_{n = 1}^\infty$ again with a straightforward diagonal argument -- we can ensure that for almost every $t > 0 $ the sequence $(\overline{u}^{\tau_n}(t))_{n = 1}^\infty$ converges to $u(t)$ in $C^2([0,1])$. From this one immediately concludes that 
  \begin{equation}
  D\mathcal{E}(\overline{u}^{\tau_n}(t))( v- \overline{u}^{\tau_n}(t)) \rightarrow D\mathcal{E}(u(t)) (v-u(t)) \quad a.e. \;  t > 0. 
  \end{equation}
  Moreover $(|D\mathcal{E}(\overline{u}^{\tau_n}(t))( v- \overline{u}^{\tau_n}(t))|)_{n = 1}^\infty$ can be dominated uniformly in $n$ by observing that by \eqref{eq:naivgrad}
 \begin{align}
| D\mathcal{E}(\overline{u}^{\tau_n}(t))( v- \overline{u}^{\tau_n}(t))| & \leq (2 ||\overline{u}^{\tau_n}(t) ||_{W^{2,2}} + 5 ||u^{\tau_n}(t)||_{W^{2,2}}^2 ) ||v- u^{\tau_n}(t)||_{W^{2,2}}^2\\ & \leq C( 1 + ||u^{\tau_n}||_{L^\infty((0,T),W^{2,2})}^3)  
 \end{align}
 which is unformly bounded by a constant because of Lemma \ref{lem:miii}. By dominated convergence we infer
 \begin{equation}\label{eq:convD}
\lim_{n\rightarrow \infty} \int_a^b D\mathcal{E}(\overline{u}^{\tau_n}(t)) (v- \overline{u}^{\tau_n}(t)) \dt  =  \int_a^b D\E(u(t)) (v-u(t)) \dt . 
 \end{equation}
 Also observe that 
 \begin{align*}
 \int_a^b ( \dot{u}^{\tau_n}(t) , v - \overline{u}^{\tau_n}(t) )  \dt & = \int_a^b (\dot{u}^{\tau_n} (t) , v - u(t)) \dt + \int_a^b (\dot{u}^{\tau_n}(t), \overline{u}^{\tau_n}(t) - u(t)) \dt 
 \end{align*}
 and note that by weak convergence of $(u^{\tau_n})_{n = 1}^\infty $ in $W^{1,2}((0,T),L^2(0,1))$ we have
 \begin{equation}\label{eq:limiueberg}
 \int_a^b (\dot{u}^{\tau_n} (t) , v - u(t)) \dt \rightarrow \int_a^b (\dot{u}(t) , v - u(t)) \dt \quad ( n \rightarrow \infty) . 
 \end{equation}
 Moreover
 \begin{equation}
  \left\vert \int_a^b (\dot{u}^{\tau_n}(t), u^{\tau_n}(t) - u(t)) \dt \right\vert \leq ||u^{\tau_n}||_{W^{1,2}((0,b),L^2(0,1))} ||u^{\tau_n} - u||_{L^2((0,b),L^2(0,1))}
 \end{equation}
which tends to zero as $n \rightarrow \infty$. This and \eqref{eq:limiueberg} together imply that 
 \begin{equation}
 \int_a^b ( \dot{u}^{\tau_n}(t) , v - u^{\tau_n}(t) ) \dt \rightarrow \int_a^b ( \dot{u}(t) , v - u(t) ) \dt
 \end{equation}
 and  together with \eqref{eq:VARDISCR} and \eqref{eq:convD} we find 
 \begin{align}
 0 & \leq \lim_{n\rightarrow \infty} \left( \int_a^b ( \dot{u}^{\tau_n}(t) , v- \overline{u}^{\tau_n}(t) ) \dt + \int_a^b D\mathcal{E}(\overline{u}^{\tau_n}(t)) (v- \overline{u}^{\tau_n}(t)) \dt\right) 
 \\ & = \int_a^b ( \dot{u}(t) , v-u(t) ) + D\mathcal{E}(u(t))( v-u(t)) \; \mathrm{d}t .
 \end{align}
 Since $a,b$ are arbitrary and the integrand lies in $L^1_{loc}((0,\infty))$ we infer that at each Lebesgue point $t$ of the integrand one has 
 \begin{equation}
 ( \dot{u}(t) , v-u(t) ) + D\mathcal{E}(u(t))( v-u(t)) \geq 0 .
 \end{equation}
 This shows \eqref{eq:endec}. It remains to show that $t \mapsto \mathcal{E}(u(t)) $ coincides almost everywhere with a nonincreasing function $f$ that satisfies $f(0)= \E(u_0)$. By Remark \ref{rem:minibound} $t \mapsto \mathcal{E}(\overline{u}^{\tau_n}(t))$  is nonincreasing for each $\tau > 0$. By Helly's theorem (cf. \cite[Lemma 3.3.3]{Ambrosio}) this sequence of functions has a pointwise limit, which is a nonincreasing function, call it $f$. We have already shown in Lemma \ref{lem:discraubinlions} that $\overline{u}^{\tau_n}(t)$ converges to $u(t)$ in $C^2([0,1])$ for almost every $t> 0$ so that $\mathcal{E}(\overline{u}^{\tau_n}(t))$ converges to $\mathcal{E}(u(t))$ pointwise almost everywhere. Hence $t \mapsto \mathcal{E}(u(t))$ coincides almost everywhere with $f$.   
  \end{proof}
  \begin{remark}\label{rem:weakconconin}
  A useful byproduct of this approach is that also $\overline{u}^{\tau_n}(t) \rightarrow u(t)$ in $L^2(0,1)$ for all $t> 0$ (and not just almost everywhere). More can be said: Boundedness of $(\overline{u}^{\tau_n}(t))_{n = 1}^\infty$ in $W^{2,2}(0,1)$ (cf. \eqref{eq:416}) implies that $ \overline{u}^{\tau_n}(t) \rightharpoonup u(t)$ weakly in $W^{2,2}(0,1)$ for all $t > 0$. 
  \end{remark}
  Finally, we have constructed an $FVI$ gradient flow. Before we can prove Theorem \ref{thm:existence} we need to discuss some further properties of the constructed flow.
  \subsection{Space regularity and Navier boundary conditions}
  The minimizing movement construction in the first part of this section is a highly nonunique concept. In general Theorem \ref{thm:existence} asserts however some additional regularity properties that hold true for every possible choice of a $FVI$ gradient flow starting at $u_0$. To show this, we will not use the above construction and work directly with the definition instead.
\begin{lemma}[Weak $W^{2,2}$-continuity in time]\label{lem:weakcon}
Let $u_0 \in C_\psi$ be such that $\mathcal{E}(u_0) < \frac{c_0^2}{4}$ and let $u$ be a $FVI$-Gradient Flow. Then $u(t) \in C_\psi$ for all $t >0 $ and for each sequence $t_n \rightarrow t$ one has $u(t_n) \rightarrow u(t)$ weakly in $W^{2,2}(0,1)$. In particular $t \mapsto u(t)$ is a bounded curve in $W^{2,2}(0,1)$. 
\end{lemma}
\begin{proof}
Let $t > 0$ be arbitrary and $u$ be as in the statement. Recall that we always identify $u$ with its $C([0,\infty),C^1([0,1])$-representative. By Definition \ref{def:fvigrad} there exists $s_n \rightarrow t$ such that $u(s_n) \in C_\psi$  and $||u(s_n)||_{W^{2,2}} \leq  ||u||_{L^\infty((0, \infty), W^{2,2})}$ for all $n \in \mathbb{N}$. We know that each subsequence of $(u(s_n))_{n = 1}^\infty$ has a subsequence that converges weakly in $W^{2,2}(0,1)$. As $u(s_n) \rightarrow u(t)$ in $C^1([0,1])$ we infer by the Urysohn property that $u(s_n) \rightharpoonup u(t)$ weakly in $W^{2,2}(0,1)$. It follows that $u(t) \in C_\psi$ and 
\begin{equation}\label{eq:allbound}
||u(t)||_{W^{2,2}} \leq \liminf_{n \rightarrow \infty} ||u(s_n)||_{W^{2,2}} \leq ||u||_{L^\infty((0, \infty), W^{2,2})} .
\end{equation}
Now let $t_n \rightarrow t$ be an arbitrary sequence.
By the choice of the representative we know that $u(t_n) \rightarrow u(t)$ in $C^1([0,1])$. By \eqref{eq:allbound}, $(u(t_n))_{n = 1}^\infty \subset W^{2,2}(0,1)$ is a bounded sequence. Therefore each subsequence has a weakly convergent subsequence in $W^{2,2}(0,1)$. Because of uniqueness of limits in $W^{1,\infty}(0,1)$ all those sequences converge weakly to $u(t)$. Again the Urysohn property yields that $u(t_n)$ converges weakly to $u(t)$ in $W^{2,2}(0,1)$. 
\end{proof}

\begin{lemma}[Space regularity and Navier boundary conditions]\label{lem:exreg}
Let $u_0 \in C_\psi$ be such that $\E(u_0) < \frac{c_0^2}{4}$. Let $u$ be any $FVI$ gradient flow starting at $u_0$. Then for almost every $t > 0$ one has that $u(t) \in W^{3,\infty} (0,1)$ and $u(t)''(0) = u(t)''(1) = 0$
\end{lemma} 
\begin{proof} Since the proof is very similar to the proof of Lemma \ref{lem:w3infbds} we only mention some important steps.
Let $t> 0$ be such that $FVI$ holds true. 
Similar to the proof of Lemma \ref{lem:w3infbds} one can infer from $FVI$ that there exists a Radon measure $\mu$ on $(0,1)$ such that for all $\phi \in C_0^\infty(0,1)$. 
\begin{equation}
\int_0^1 \dot{u}(t) \phi \dx + D\E(u(t))( \phi) = \int_0^1 \phi \; \mathrm{d} \mu 
\end{equation}
and for all $\phi \in W^{2,2}(0,1) \cap W_0^{1,2}(0,1)$ such that $\mathrm{supp}(\phi)$ is compactly contained in $\{ u > \psi\}$
\begin{equation}
\int_0^1 \dot{u}(t) \phi \dx + D\E(u(t))( \phi) = 0.
\end{equation}
Proceeding similar to the proof of Lemma \ref{lem:w3infbds} we can derive the claimed regularity and the Navier boundary conditions.  
\end{proof}

\begin{proof}[Proof of Theorem \ref{thm:existence}]
Existence of $u$ follows from Corollary \ref{cor:existiert}. That $t \mapsto u(t) \in W^{2,2}(0,1)$ is everywhere defined and bounded follows from the last sentence of Lemma \ref{lem:weakcon}. The additional space regularity and the Navier boundary conditions follow from Lemma \ref{lem:exreg}. 
\end{proof}

\subsection{Energy dissipation}

In the rest of this section we will prove an \emph{energy dissipation inequality}. This shows that energy is dissipated in $(0,T)$ is comparable to $||\dot{u}||_{L^2((0,T),L^2(0,1))}^2$, which is what one would expect for a gradient flow. The speed of energy dissipation we obtain might however be worse than in the usual formulation of metric gradient flows. The expected dissipation speed can be described by De Giorgi's energy dissipation identity, cf. \cite[Section 2.3]{Marius2}. How much worse the FVI gradient flow performs depends highly on the quantity $|\partial^- \E|$ from \cite[Equation (2.3.1)]{Ambrosio}, cf. \cite[Theorem 2.3.3]{Ambrosio}. 

\begin{lemma}[An energy dissipation inequality]
Let $u_0 \in C_\psi$ be such that $\E(u_0) < \frac{c_0^2}{4}$ 
Let $u$ be an $FVI$ gradient flow starting at $u_0 $ which was constructed as in the Proof of Theorem \ref{thm:existence}. Then for each $T>0$ one has 
\begin{equation}
\E(u(T)) + \frac{1}{2}\int_0^T ||\dot{u}(t)||_{L^2}^2 \dt \leq \E(u_0) .  
\end{equation}
\end{lemma}

\begin{proof}
Let $(u^{\tau_n})_{n = 1}^\infty$ be the sequence from Lemma \ref{lem:a.e}. For $n \in \mathbb{N}$ we define $k_n \in \mathbb{N}_0$ to be the unique integer such that $k_n \tau_n \leq T \leq (k_n + 1) \tau_n$.  By weak convergence of $u^{\tau_n}$ in $W^{1,2}((0,T),L^2(0,1))$ (cf. Lemma \ref{lemma:limitcase}) and weak $W^{2,2}-$convergence of $\overline{u}^{\tau_n}(T)$ to $u(T)$ (cf. Remark \ref{rem:weakconconin}) we obtain with \eqref{eq:stanminimov}
\begin{align}
\E(u(T)) + & \frac{1}{2} \int_0^T ||\dot{u}(t)||_{L^2}^2 \dt  \leq \liminf_{n  \rightarrow \infty} \left( \E(\overline{u}^{\tau_n}(T)) + \frac{1}{2} \int_0^T ||\dot{u}^{\tau_n}(t)||^2 \dt \right) 
\\ & \leq \liminf_{n \rightarrow \infty } \left( \E(u_{(k_n+1) \tau_n}) + \frac{1}{2}\int_0^{(k_n+1) \tau_n} || \dot{u}^{\tau_n}(t) ||^2 \dt \right) 
\\ & \leq \liminf_{n \rightarrow \infty } \left( \E(u_{(k_n+1) \tau_n}) + \sum_{l =0}^{k_n} \frac{||u_{(l+1)\tau_n} - u_{l\tau_n}||_{L^2}^2}{2\tau} \right) 
\\ & \leq \liminf_{n \rightarrow \infty } \left( \E(u_{(k_n+1) \tau_n}) + \sum_{l =0}^{k_n} (\E(u_l)- \E(u_{l+1}))  \right) = \E(u_0).  \qedhere
\end{align}
\end{proof}

  \subsection{Uniqueness and preservation of symmetry}
  
Now that we have shown existence of $FVI$ gradient flows one can ask whether they are unique. This uniqueness has been obtained in \cite[Section 3]{Yoshizawa}. It has an important consequence for our later studies of the asymptotics --- namely that evolutions are symmetry preserving, as we shall show. 



\begin{prop}[{Uniqueness, cf. \cite[Theorem 3.2]{Yoshizawa}}] 
Suppose that $u_0\in C_\psi$ is such that $\mathcal{E}(u_0)< \frac{c_0^2}{4}$. 
Then the $FVI$ gradient flow starting at $u_0$ is unique. 
\end{prop}
\begin{proof}
This has been shown \cite[Theorem 3.2]{Yoshizawa} for a \emph{length-penalized} elastic energy $\mathcal{E}+ \lambda \mathcal{L}$, $\lambda > 0$. By \cite[Remark 6.6]{Yoshizawa} however the case $\lambda=0$ can also be shown following the lines of  \cite[Section 3]{Yoshizawa} provided that $\mathcal{E}(u_0) < \frac{c_0^2}{4}$. This energy estimate is needed in the same way as in the existence proof, namely for the control in Remark \ref{ref:remeng}. 
\end{proof}

\begin{cor}[Symmetry preservation]\label{cor:sympres}
Suppose that $\psi$ is symmetric, i.e. $\psi (1- \cdot) = \psi$. Suppose that $u_0 \in C_\psi$ is symmetric and such that $\E(u_0) <  \frac{c_0^2}{4}$. Let $u$ be the $FVI$ gradient flow starting at $u_0$. Then $u(t) = u(t)(1- \cdot)$ for all $t> 0$.    
\end{cor}
\begin{proof}
Let $u$ be as in the statement. We show that $\widetilde{u} : t \mapsto u(t)(1-\cdot)$ is an $FVI$ gradient flow. As $u$ and $\widetilde{u}$ have the same initial datum, they must coincide by uniqueness. From the symmetry of $\psi$ follows that $\widetilde{u}(t) \in C_\psi$ for almost every $t> 0$. The regularity requirements are also easily to be checked. Moreover, by symmetry of $\E$, $\E \circ \widetilde{u} = \E \circ u $ coincides almost everywhere with a nonincreasing function that takes the value $\E(u_0)$ at $ t= 0$.

 To verify the $FVI$ equation we first observe by direct computation that for all $u, \phi \in W^{2,2}(0,1) \cap W_0^{1,2}(0,1)$ one has
\begin{equation}\label{eq:syyymgrad}
D\E(u)(\phi) = D\E(u(1-\cdot))(\phi(1- \cdot)) .
\end{equation}
For arbitrary $v \in C_\psi$ we infer by symmetry properties of the $L^2$ scalar product that 
\begin{align}
& (\dot{\widetilde{u}}(t), v - \widetilde{u}(t) ) + D\E (\widetilde{u}(t)) (v- \widetilde{u}(t) ) 
\\ & = (\dot{u}(t)(1- \cdot), v - u(t)(1-\cdot) ) + D\E ( u(t)(1- \cdot) )( v- u(t)(1- \cdot) ) 
\\ & = ( \dot{u}(t), v(1- \cdot)- u(t)) + D\E(u(t))(v (1- \cdot) - u(t) ) \geq 0 ,
\end{align}
for $a.e. \; t > 0$, because $u$ is an $FVI$ Gradient Flow and $v(1- \cdot) \in C_\psi$ because of the symmetry of the obstacle.  
\end{proof}
\section{Qualitative Behavior}
Describing the qualitative behavior of higher order PDEs is in general a challenging task as there is no maximum priciple available that would allow a comparision of solutions. In the field of parabolic obstacle problems one is however interested in several qualitative aspects, in particular the description of the coincidence set $\{ u(t) = \psi\}$ that forms the now time-dependent free boundary of the problem. 
\subsection{The coincidence set}
Here we prove that the obstacle is touched in finite time, provided that the initial energy is suitably small. Not much more can be said about the size of the coincedence set as there exist critical points for which the coincedence set is only a singleton (cf. \cite[Proposition 3.2]{Marius1}). 
\begin{proof}[Proof of Proposition \ref{prop:touch}]
Suppose that $\mathcal{E}(u_0) < G( \sqrt{\frac{2}{3}})^2$. Observe that then 
\begin{equation}
3 - \frac{5}{(1 + G^{-1}(\sqrt{\E(u_0))})^2} < 0 .
\end{equation} 
From Remark \ref{ref:remeng} also follows that $\inf_{u \in C_\psi} \E(u) = \min_{ u \in C_\psi} \E(u)  > 0$. We here prove the slightly stronger statement that each time interval of length larger than 
\begin{equation}
L_0 := \frac{ G^{-1} ( \sqrt{\E(u_0)} )^2 }{2\inf_{u \in C_\psi}\E(u)} \frac{1}{ \frac{5}{(1 + G^{-1}(\sqrt{\E(u_0))})^2} - 3} 
\end{equation}
must contain a time $t$ such that $u(t)$ touches $\psi$. Suppose that $(a,b)$ is an interval of length exceeding $L_0$ such that $\{ u(t) = \psi \} = \emptyset $ on $(a,b)$. Note that then $u(t) > \psi$ and Proposition \ref{prop:concistncy} yields that  
\begin{equation}\label{eq:ELREAL}
(\dot{u}(t) , \phi ) + D\E(u(t)) (\phi) = 0 \quad  \forall \phi \in W^{2,2}(0,1) \cap W_0^{1,2}(0,1)
\end{equation}
for almost every $t \in (a,b)$. 
We use again the Lions-Magenes-Lemma to compute in the sense of distributions we have
\begin{align}
\frac{d}{dt} \int_0^t u(t)^2 \dx = 2 ( \dot{u}(t) , u(t) ) .
\end{align} 
Since $t \mapsto u(t)$ is absolutely continuous with values in $L^2$, so is $t \mapsto ||u(t)||_{L^2}$ with values in $\mathbb{R}$. By the product rule for Sobolev functions and the fact that $t \mapsto ||u(t)||_{L^2}$ is uniformly bounded in $t$, $t \mapsto ||u(t)||_{L^2}^2$ lies in $W^{1,1}(0,1)$. Hence the above inequality holds also pointwise almost everywhere and the fundamental theorem of calculus can be applied. By Theorem \ref{thm:existence}, $u(t) \in C_\psi \cap  W^{3,2}(0,1)$ and $u(t)''(0) = u(t)''(1) = 0 $ for almost every $t$. For those $t$ we can define $A_{t} := \frac{u(t)''}{(1+u(t)'^2)^\frac{5}{4}}$ and use \eqref{eq:ELREAL} and \eqref{eq:euler} to find 
\begin{align}
\frac{d}{dt} \int_0^1 u(t)^2 \dx  & = 2 \int_0^1 u(t) \dot{u}(t) \dx = 4 \int_0^1 \frac{A_t'}{(1+ u(t)'^2)^\frac{5}{4}} u(t)' \dx 
\\ &  = \left[ \frac{4A_t u(t)' }{(1+ u(t)'^2)^\frac{5}{4}} \right]_{x = 0}^{x= 1} - \int_0^1 4A_t \left( \frac{u(t)''}{(1+u(t)'^2)^\frac{5}{4}} - \frac{5}{2} \frac{u(t)'^2u(t)''}{(1+ u(t)'^2)^\frac{9}{4}} \right) \dx  
\\  & = 4  \left( -\int_0^1 A_t^2 \dx  +  \frac{5}{2} \int_0^1 A_t^2 \dx - \frac{u(t)''^2}{(1+ u(t)'^2)^\frac{7}{2}} \dx \right) 
\\ & \leq 4 \left( \frac{3}{2} \E(u(t)) - \frac{5}{2} \int_0^1 \frac{u(t)''^2}{(1+u(t)'^2)^\frac{7}{2}} \dx  \right) 
\\ & \leq2  \left( 3 - \frac{5}{(1+G^{-1}(\sqrt{\E(u_0)})^2} \right) \mathcal{E}(u(t)) \\ &  \leq 2  \left( 3 - \frac{5}{(1+G^{-1}(\sqrt{\E(u_0)})^2} \right) \inf_{u \in C_\psi} \E(u) ,
\end{align}
which is negative by the assumptions.
By the fundamental theorem of calculus (whose applicability we have discussed above) and Remark \ref{ref:remeng} we find 
\begin{align*}
\int_0^1 u(b)^2 \dx  & \leq \int_0^1 u(a)^2 \dx  + 2  \left( 3 - \frac{5}{(1+G^{-1}(\sqrt{\E(u_0)})^2} \right) \inf_{u \in C_\psi} \E(u) (b-a)\\ &  < ||u(a)'||^2_{L^\infty} + L_0 2  \left( 3 - \frac{5}{(1+G^{-1}(\sqrt{\E(u_0)})^2} \right) \inf_{u \in C_\psi}\E(u) 
\\ & \leq G^{-1}(\sqrt{\E(u_0)})^2+ L_0 2  \left( 3 - \frac{5}{(1+G^{-1}(\sqrt{\E(u_0)})^2} \right) \inf_{u \in C_\psi}\E(u) = 0 .
\end{align*}
which results in a contradiction as the expression on the left hand side must be nonnegative. 
\end{proof}
\subsection{Time regularity}
Since the constructed evolution is not driven by an equation but rather by an inequality one can not immediately obtain time regularity from space regularity. In general, time regularity for parabolic obstacle problems is an important problem. A technique that has been applied in previous works, e.g. \cite{Novaga1}, is to consider the flow as singular limit of perturbed evolutions without obstacle. We refer to \cite{Caffarelli} for a discussion of this technique. We remark that this approach heavily relies on uniqueness which is not the focus of this article. This is why we present a different approach.

\begin{lemma}[Time continuity in energy space]
Let $u_0 \in C_\psi$ be such that $\mathcal{E}(u_0) < \frac{c_0^2}{4}$ and let $u$ be the $C([0, \infty), C^1([0,1]) )$-representative of an $FVI$-Gradient Flow starting at $u_0$. Let 
\begin{equation}
A:= \{ s \in [0,\infty) : \lim_{r\rightarrow s} u(r)  = u(s)  \; \textrm{in} \;  W^{2,2}(0,1) \}
\end{equation}
the set of points of $W^{2,2}$-continuity of $u$. Then $|[0, \infty) \setminus A| = 0 $. 
\end{lemma}
\begin{proof}
Since $u$ is an $FVI$ Gradient Flow, there exists a nonincreasing function $\phi : [0,\infty) \rightarrow \mathbb{R}$ and a set $\widetilde{B} \subset [0,\infty)$ such that $|[0, \infty) \setminus \widetilde{B}| = 0 $
and  $\mathcal{E} \circ u = \phi$ on $\widetilde{B}$. Also, let $B$ be all the points of continuity of $\phi$ that lie in $\widetilde{B}$. Since $\phi$ is monotone, $\widetilde{B} \setminus B$ is at most countable and we find that $|[0,\infty) \setminus B| = 0$. 
 Moreover define 
\begin{equation*}
E:= \{ s: s \textrm{ is not a Lebesgue point of} \; ||\dot{u}||_{L^2}, \; \textrm{or $FVI$ does not hold true at s} \}.
\end{equation*}
Since $|E|=0$ it suffices to show that each point $s \in B \setminus E$ is a point of $W^{2,2}$-continuity. We fix therefore $s \in  B \setminus E$ and let first $t> 0$ be arbitrary. All we know then is that $u(t)\in C_\psi$ by Lemma \ref{lem:weakcon}. Now we compute, again using that $x \mapsto \frac{1}{(1+x^2)^\frac{5}{2}}$ is Lipschitz continuous and Remark \ref{ref:remeng}
\begin{align*}
& \frac{1}{( 1 + G^{-1} (\sqrt{\E(u_0)})^2) ^\frac{5}{2}}  \int_0^1 ( u(t)'' - u(s)'')^2  \dx \leq \int_0^1 \frac{(u(t)'' - u(s)'')^2}{(1+ u(s)'^2)^\frac{5}{2}} \dx
\\ & = \int_0^1 \frac{u(t)''^2}{(1+ u(s)'^2)^\frac{5}{2}} \dx - \int_0^1 \frac{u(s)''^2}{(1+ u(s)'^2)^\frac{5}{2}} \dx  +   2 \int_0^1 \frac{u(s)'' ( u(s)''- u(t)'')}{(1 + u(s)'^2)^\frac{5}{2}} \dx
\\ & = \int_0^1 \frac{u(t)''^2}{(1+u(t)'^2)^\frac{5}{2}} \dx + \int_0^1 u(t)''^2 \left[ \frac{1}{(1 + u(s)'^2)^\frac{5}{2}} - \frac{1}{(1 + u(t)'^2)^\frac{5}{2}}  \right] \dx \\ & \quad  -  \int_0^1 \frac{u(s)''^2}{(1+ u(s)'^2)^\frac{5}{2}} \dx  +  D\E(u(s))(u(s) - u(t)) \\ & \quad + 5 \int_0^1 \frac{u(s)''^2 u(s)'}{(1 +u(s_n)'^2)^\frac{7}{2}} (u(s)' - u(t)' )  \dx 
\\ & \leq \E(u(t)) - \E(u(s))     + \frac{5}{2} \E(u_0) ( 1 + G^{-1}(\sqrt{\E(u_0)}) ^2)^\frac{5}{2} ||u(s)' - u(t)'||_{L^\infty} \\ &   \quad + ( \dot{u}(s), u(t) - u(s) ) + 5 \E(u(s)) || u(s) ' - u(t)'||_{L^\infty} 
\\ & \leq  \E(u(t)) - \E(u(s)) + (D+ ||\dot{u}(s)||_{L^2}) || u(s) ' - u(t)'||_{L^\infty}
\end{align*}
where $D> 0$ is an appropriately chosen constant that does not depend on $t$. We find that  there exists $C_0 > 0 $ such that
\begin{equation}\label{eq:contii}
||u(t)- u(s)||_{W^{2,2}} \leq C_0(\E(u(t)) - \phi(s)) + (D+ ||\dot{u}(s)||_{L^2}) || u(s) ' - u(t)'||_{L^\infty})
\end{equation}
for all arbitrary $t > 0.$ Now let $\epsilon > 0 $ be arbitrary. Since $s$ is a point of continuity of $\phi$ there exists $\delta_1 > 0 $ such that $\sup_{t \in B_{\delta_1} ( s) } |\phi(t) - \phi(s) | < \frac{\epsilon}{2C_0}$. Moreover  $||\dot{u}(s) ||_{L^2}< \infty$ as $s$ is a Lebesgue point of $\dot{u}$ and therefore there exists $\delta_2 > 0 $ such that 
$\sup_{ t \in B_{\delta_2} (s) } || u(t)' - u(s)'||_{L^\infty} < \frac{\epsilon}{2C_0( D + ||\dot{u}(s) ||) }$.  Now choose $\delta := \min\{\frac{\delta_1}{2}, \delta_2 \}$. Let $t \in (0,\infty)$ be such that $|t -s | <\delta$. Then there exists a sequence $t_n \rightarrow t $ such that $(t_n)_{n = 1}^\infty \subset B$ as $|(0, \infty) \setminus B| = 0$. We can assume without loss of generality that for all $n \in \mathbb{N}$ one has $|t_n - t| < \frac{\delta_1}{2}$, which implies $|t_n - s| < \delta_1$ for all $n \in \mathbb{N}$. Now note that by weak lower semicontinuity of $\E$ (cf. \cite[Proof of Lemma 2.5]{Anna}) and Lemma \ref{lem:weakcon} we have 
\begin{equation}
(\E(u(t)) - \phi(s)) \leq \liminf_{n \rightarrow \infty} (\E(u(t_n)) - \phi(s)) \leq \liminf_{n \rightarrow \infty} (\phi(t_n) - \phi(s)) \leq \frac{\epsilon}{2C_0}  
\end{equation}
by the choice of $\delta_1$. This and \eqref{eq:contii} imply that 
\begin{equation}
||u(t) - u(s)||_{W^{2,2}} < \epsilon. \qedhere
\end{equation}
\end{proof}

\section{Critical Points}
In the next section we want to examine the critical points of $\mathcal{E}$ in $C_\psi$. 
One question that could be asked is how many critical points exist. A partial answer is given in \cite[Corollary 5.22]{Marius2} and \cite{Miura3}, where it is shown that there exist no critical points above an obstacle of a certain height. This is also why our convergence results may only hold true for small obstacles. 

Once existence is ensured, another question one can look at is symmetry of critical points, which is to expect since the equation has a symmetry: If $u\in C_\psi$  solves \eqref{eq:symm} and $\psi= \psi(1- \cdot)$ then also $u(1- \cdot) \in C_\psi$ is a solution of \eqref{eq:symm}, as follows directly from \eqref{eq:syyymgrad}.

\begin{lemma}[Regularity and concavity of critical points]\label{lem:critii}
Let $u \in C$ be a critical point. Then $u \in W^{3,\infty}(0,1), u''(0) = u''(1)= 0$ and $u$ is concave. Moreover, if $u > \psi$ on some interval $(a,b)$ then $u_{\mid_{[a,b]}} \in C^\infty([a,b])$ and  
\begin{equation}\label{eq:72}
\frac{A_u'}{(1+u'^2)^\frac{5}{4}} \equiv \mathrm{const.} \; on \; (a,b). 
\end{equation}
\end{lemma}
\begin{proof}
For the regularity, \eqref{eq:72} and the fact that $u''(0) = u''(1) = 0$, we refer to \cite[Corollary 3.2 and Theorem 5.1]{Anna}. For the concavity observe that by \eqref{eq:euler} 
\begin{equation}\label{eq:73}
\int_0^1 \frac{A_u' \phi'}{(1+u'^2)^\frac{5}{4}} \leq 0  \quad \forall \phi \in W^{2,2}(0,1) \cap W_0^{1,2}(0,1): \phi \geq 0 .
\end{equation}
By density we obtain that the same holds true for all  $\phi \in W_0^{1,2}(0,1)$ such that $\phi \geq 0$. Plugging in $\phi = \max\{A_u,0\}$, which is admissible as by the previous regularity $A_u \in W^{1,\infty}(0,1)$ and $A_u(0) = u''(0) = 0$ and $A_u(1)= u''(1) = 0$. We obtain that 
\begin{equation}
\int_0^1 \frac{\max\{ A_u , 0 \}'^2}{(1 + (u')^2)^\frac{5}{4}} \leq 0 .
\end{equation}
This implies that $\max\{ A_u ,0 \} = 0 $ a.e. and hence $A_u \leq 0 $ a.e.. In particular we can conclude that $u'' \leq 0 $ a.e. which implies the concavity of $u$. 
\end{proof}
\begin{remark}\label{rem:posi}
Note that concavity of critical points implies in particular that those are nonnegative, i.e. $u \geq 0$. 
\end{remark}
\subsection{Symmetry of minimizers}
Critical points of special importance are minimizers, which exist by Remark \ref{ref:remeng} whenever $\inf_{u \in C_\psi} \E(u) < \frac{c_0^2}{4}$. Here we investigate symmetry of those. The main method used will use is a nonlinear version of Talenti's inequality for which we need some additional notation. For $f \in L^1(0,1)$ we denote by $\mu_f(t) := |\{ f > t \}|$ and by $f^*(x) := \inf \{ t >0 : \mu_f(t) < x \}$. Moreover we define $f_*(x) := f^*( 2 | x- \frac{1}{2}|)$ and call $f_*$ the \emph{symmetric decreasing rearrangement} of $f$. Note that for each decreasing function $g: (0,1) \rightarrow (0,1)$ one has that $r(x) := g( 2 |x-\frac{1}{2}| ) $ satisfies $r_* = r$. Another important fact is that $||f||_{L^p} = ||f^*||_{L^p} = ||f_*||_{L^p}$ for each $p \in [1, \infty]$, cf. \cite[Section 3.3]{LiebLoss}.

The proof of the next result can be regarded as a special case of \cite[Theorem 1]{Talenti} in one dimension. Since the assumptions in this article differ however slightly from our situation we give a self-contained proof, which however follows the lines of the proof in \cite{Talenti}.
\begin{lemma}[A nonlinear version of Talenti's symmetrization result] \label{lem:talenti}
Let $H \in C^\infty(\mathbb{R})$ be an odd function that satisfies $H' > 0 $. Moreover, let $f \in L^2(0,1)$ be nonnegative and such that $ \frac{1}{2}||f||_{L^2(0,1)} \leq ||H||_\infty$. 
Suppose that $u \in W^{2,2}(0,1)$ is a nonnegative weak solution of 
\begin{equation}\label{eq:conzi}
\begin{cases}
-H(u')' = f & \textrm{in} \;  (0,1), \\
u(0) = u(1) = 0 .
\end{cases}
\end{equation}
Then there exists a unique symmetric weak solution $v \in W^{2,2}(0,1)$ of 
\begin{equation}\label{eq:Talenti}
\begin{cases}
-H(v')' = f_* & \; \textrm{in}\;  (0,1), \\
v(0) = v(1) = 0.
\end{cases}
\end{equation}
If $\frac{1}{H^{-1}}$ is convex on $[0,H(||u'||_\infty)]$ then one has $v \geq u_*$.
\end{lemma}
\begin{proof}Without loss of generality $u \not \equiv 0 $, otherwise the claim is trivially true. 
To show the existence of $v$ we set 
\begin{equation}\label{eq:symexpl}
v(x) := \frac{1}{2} \int_{ 2 |x-\frac{1}{2}|}^1 H^{-1} \left( \frac{1}{2} \int_0^s f^*(r) \dr \right) \ds.  
\end{equation}
Note that $v$ is well-defined because of the fact that $\frac{1}{2} ||f^*||_{L^2(0,1)} =\frac{1}{2}||f||_{L^2(0,1)}    \leq ||H||_\infty$ and inverse function $H^{-1}$ is defined on $ H(\mathbb{R}) =  ( -||H||_\infty, ||H||_\infty )$.  Symmetry and \eqref{eq:Talenti} follow by direct computation. 
For the uniqueness suppose that $v_1,v_2$ are symmetric weak solutions of \eqref{eq:Talenti}. It follows that $H(v_1'(x))- H(v_2'(x)) = \mathrm{const.}$ on $(0,1)$. Plugging in $x= \frac{1}{2}$ and using that by symmetry $v_1'(\frac{1}{2}) = v_2'(\frac{1}{2}) = 0$ we find that the constant on the right hand side equals zero and hence $H(v_1') = H(v_2')$. As $H$ is by assumption invertible we obtain $v_1' = v_2'$ and now the fact that $v_1(0) = v_2(0)= 0 $ implies the claim. For the Talenti-type inequality we define as in \cite{Talenti}
\begin{equation}
 \Phi(t) := \int_{ \{ u > t \}} H(u') u' \dx.
\end{equation}  
which is nonnegative and nonincreasing in $t$ as $H(u')u' = H(|u'|)|u'| \geq  0$. Hence $\Phi$ is almost everywhere differentiable. Let $t$ now be a point of differentiability of $\Phi$. Note that
\begin{equation}
\Phi(t) = \int_0^1 H(u') ( \max(u-t, 0 ) )' \dx = \int_0^1 f(x) \max(u(x)-t,0) \dx. 
\end{equation}
Observe that then for each $h > 0 $ 
\begin{equation}
 \Phi(t) - \Phi(t+h) = \int_0^1 f(x) (\max(u-t,0)- \max(u-t-h, 0) ) \dx \leq h \int_{\{u > t\} } f(x) dx .
\end{equation}
By \cite[Equation (2.6b)]{Talenti} we obtain 
\begin{equation}\label{eq:26b}
- \Phi'(t) \leq \int_{\{u > t\} } f(x) dx  \leq \int_0^{\mu_u(t)} f^*(r) \dr .
\end{equation}
By \cite[Equation (2.22)]{Talenti}) we have 
\begin{equation}
-\frac{d}{dt} \int_{ \{ u > t \} } |u'| \dx \geq 2 \quad a.e. \; t. 
\end{equation} 
By monotonicity of $H^{-1}$ and by Jensen's inequality, which is applicable because of the convexity assumption on $\frac{1}{H^{-1}}$ we obtain that for almost every $t$ 
\begin{align}
 \frac{1}{H^{-1}} \left( \frac{-\Phi'(t)}{2} \right)& \leq  \frac{1}{H^{-1}} \left( \frac{-\Phi'(t)}{- \frac{d}{dt} \int_{ \{u > t \} } |u'| \dx  } \right) 
 = \lim_{h \rightarrow 0+ }  \frac{1}{H^{-1}} \left( \frac{\Phi(t) - \Phi(t+h)}{ \int_{ \{t < u \leq t+h\} } |u'| \dx }\right)
\\ & = \lim_{h \rightarrow 0+ }  \frac{1}{H^{-1} }\left( \frac{\int_{ t < u \leq t+ h } H(u') u' \dx }{ \int_{ \{t < u \leq t+h\} } |u'| \dx }\right)
\\ & = \lim_{h \rightarrow 0+ }  \frac{1}{H^{-1} }\left( \frac{\int_{ \{t < u \leq t+ h \} } H(|u'|) |u'| \dx }{ \int_{ \{t < u \leq t+h\} } |u'| \dx }\right)
\\ & \leq  \lim_{h \rightarrow 0+ }  \frac{\int_{ \{ t < u \leq t+ h \} } 1 \dx }{ \int_{ \{t < u \leq t+h\} } |u'| \dx} = \lim_{h \rightarrow 0 +}  \frac{\mu_u(t) - \mu_u(t+h)}{\int_{ \{t < u \leq t+h\} } |u'| \dx} \\ &  = \frac{-\mu_u'(t)}{- \frac{d}{dt} \int_{ \{u > t \} } |u'| \dx }  \leq - \frac{\mu_u'(t)}{2}.
\end{align}
Hence for almost every $t > 0 $ we have by the previous computation and by \eqref{eq:26b} 
\begin{equation}\label{eq.Numbbaa}
1 \leq  - \frac{\mu_u'(t)}{2} H^{-1} \left(  \frac{-\Phi'(t)}{2   } \right) \leq  - \frac{\mu_u'(t)}{2} H^{-1} \left(  \frac{1}{2   } \int_0^{\mu_u(t)} f^*(r) \dr  \right) .
\end{equation}
Now define 
\begin{equation}
W(t) := \int_{\mu_u(t)}^1 \frac{1}{2} H^{-1} \left( \frac{1}{2} \int_0^s f^*(r) \dr  \right) \ds.
\end{equation}
Note that $W$ is increasing. With the mean value theorem for integrals it can be shown that $W$ is differentiable at $t$ at all points of differentiability of $\mu_u$ and at all such points \eqref{eq.Numbbaa} yields $W'(t) \geq 1$. 
 By \cite[Proposition 4.7]{Chae} 
 \begin{equation}
 t = \int_0^t 1 \ds \leq \int_0^t W'(s) \ds \leq W(t) - W(0) .
 \end{equation}
 Note that $W(0) = 0 $ as $\mu_u(0) = 1$. This is so since $u$ is by \eqref{eq:conzi} concave and nonnegative and therefore $\{ u= 0 \} = \{0,1 \}$ or $u \equiv 0$ where we excluded the last case in the beginning of the proof. Hence $ \{u =0 \}$ is a Lebesgue null set and therefore $|\{ u > 0 \}| = 1.$ We obtain that $ t \leq W(t)$, i.e. 
 \begin{equation}
 t \leq \int_{\mu_u(t)}^1 \frac{1}{2} H^{-1} \left( \frac{1}{2} \int_0^s f^*(r) \dr  \right)\ds  .
 \end{equation}
By the very definition of $u^*$ we get that 
\begin{equation}
u^*(x) \leq \int_x^1 \frac{1}{2} H^{-1} \left( \frac{1}{2} \int_0^s f^*(r) \dr  \right) \ds.
\end{equation}
 Finally 
 \begin{equation}
 u_*(x) = u^*\left( 2 \left\vert x- \nicefrac{1}{2} \right\vert \right) \leq   \int_{2 |x- \frac{1}{2}|}^1 \frac{1}{2} H^{-1} \left( \frac{1}{2} \int_0^s f^*(r) \dr  \right) = v(x),
 \end{equation}
 where we used \eqref{eq:symexpl} in the last step. 
\end{proof}
\begin{cor}[Symmetry of minimizers]\label{cor:minisym}
Suppose that $\psi$ is symmetric and radially decreasing, i.e. $\psi_* = \psi$. If $\inf_{u \in C_\psi} \mathcal{E}(u)< G(2)^2$ then there exists a symmetric minimizer of $\mathcal{E}$.
\end{cor}
\begin{proof}
Let $u \in  C_\psi$ be a minimizer, which exists by Remark \ref{ref:remeng} as $G(2)^2 < \frac{c_0^2}{4}$. Note also by Remark \ref{ref:remeng} that $||u'||_\infty \leq 2$. Moreover $u$ is concave and nonnegative by Lemma \ref{lem:critii} and Remark \ref{rem:posi}. Hence $u\in C_\psi$ is a nonnegative solution of 
\begin{equation}
\begin{cases}
-[G(u')]' = -[G(u')]' & \textrm{on } (0,1) \\ 
u(0) = u(1) = 0 .
\end{cases}
\end{equation}
Observe that $f := -G(u')' $ is nonnegative as $f = - \frac{u''}{(1+ u'^2)^\frac{5}{4}} \geq 0 $ almost everywhere due to concavity of $u$. Also observe that  $||f||_{L^2}^2  = \inf_{w \in C_\psi} \E(w)$. Now define $v \in W^{2,2}$ to be the unique solution of 
\begin{equation}
\begin{cases}
-[G(v')]' = f_* & \textrm{on } (0,1), \\ 
v(0) = v(1) = 0 .
\end{cases}
\end{equation}
We will now use Lemma \ref{lem:talenti} to deduce that $v \geq (u)_* \geq \psi_* = \psi$. To apply Lemma \ref{lem:talenti} we have to check that  $\frac{1}{2}||f||_{L^2} < ||G||_\infty$ and that 
$\frac{1}{G^{-1}}$ is convex on $[0, G(||u'||_\infty)]$. For the $L^2$-bound we can look at 
\begin{equation}
\frac{1}{2}||f||_{L^2} \leq \frac{1}{2} \sqrt{\E(u)} \leq G(2) < ||G||_\infty. 
\end{equation}
 For the convexity of $\frac{1}{G^{-1}}$ we can compute for arbitrary $s \in (0,||G||_\infty)$ that
 \begin{equation}
 \left( \frac{1}{G^{-1}(s)} \right)'' = \frac{2-\frac{1}{2}G^{-1}(s)^2}{G^{-1}(s)^3} ( 1+ G^{-1}(s)^2)^\frac{3}{2}
 \end{equation}
 which makes $\frac{1}{G^{-1}}$ convex on $[0, G(2)]$. Note that $||u'||_\infty \leq2$ implies $G(||u'||_\infty) \leq G(2)$ and hence  $\frac{1}{G^{-1}}$ is convex in $[0, G(||u'||_\infty)]$. Thus Lemma \ref{lem:talenti} is applicable and we find that $v \in C_\psi$ is admissible and symmetric. Moreover
 \begin{equation}
 \mathcal{E}(v) = ||f_*||_{L^2}^2 = ||f||_{L^2}^2 = \mathcal{E}(u) = \inf_{w \in C_\psi} \E(w) ,
 \end{equation}
 which implies that $v \in C_\psi$ is  another minimizer. 
\end{proof}
\subsection{Uniqueness of symmetric critical points}

We show now uniqueness of critical points for symmetric cone obstacles. This will follow from a more general uniqueness result for solutions to ODEs that we will prove in the appendix. 
\begin{lemma}[Uniqueness of strictly concave solutions, Proof in Appendix \ref{ref:appc}]\label{lem:75}
Let $x_0 > 0$ and let $J : [0,x_0] \rightarrow \mathbb{R}$ be  nonnegative and decreasing such that $J > 0$ on $(0,x_0)$ Further assume that $J$ is locally Lipschitz continuous on $(0,x_0)$ and $J(x_0) = 0$. Then there exists at most one solution $f \in C^2([0, \frac{1}{2}])$ to  
\begin{equation}
\begin{cases}
f'(r) = J(f(r)) & r \in [0, \frac{1}{2}], \\
f(\frac{1}{2}) = x_0, f(0) = 0,  \\
\textrm{f is strictly concave on $(0, \frac{1}{2}]$}.
 \end{cases} 
\end{equation}
Here we call a $C^2$-function strictly concave on a set $A$ if $f''< 0 $ on $A$.  
\end{lemma}
The previous lemma is inspired by the following observation: A primary example for nonuniqueness of solutions to initial value problems is
\begin{equation}
\begin{cases}
\dot{x}(t) = 2 \sqrt{|x(t)|}, \\ x(0) = 0.
\end{cases}
\end{equation}
It possesses infinitely many solutions but only one of them, namely $t \mapsto t^2$, is strictly convex in $(0, \infty)$, cf. Figure \ref{fig:fig3}. 
\begin{figure}
\includegraphics[scale=0.6]{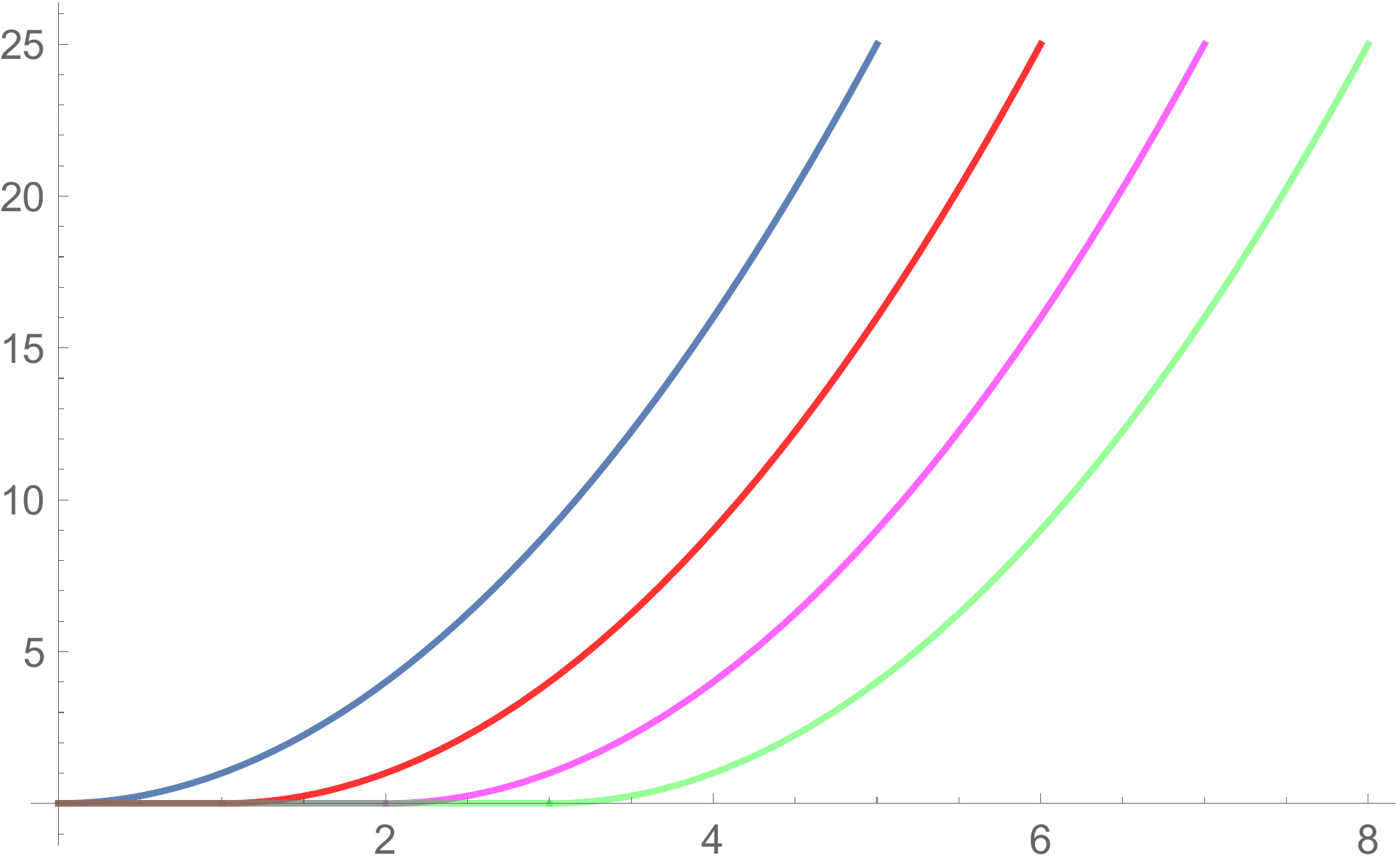}
\caption{The solution with initial value $0$ is not unique, but only one of the solutions is stricly convex.}
\label{fig:fig3}
\end{figure}

The following analysis of critical points has been obtained independently in \cite{Yoshizawa2}.

\begin{lemma}[Symmetric cone obstacles]\label{lem:symcone} 
Let $\psi \in C([0,1])$ be a symmetric cone obstacle, i.e. $\psi$ is symmetric and affine linear on $[0,\frac{1}{2}]$. Then there exists at most one symmetric constrained critical point of $\mathcal{E}$ in $C$. 
\end{lemma}
\begin{proof}
Let $u$ be a symmetric critical point. We will in the following derive an explicit formula for $u$ that characterizes it uniquely.
We claim first that $\{ u = \psi \} =  \{ \frac{1}{2} \} $. In case that $u(a) = \psi(a)$ for some $a \in (0, \frac{1}{2})$ one gets  $u'(a) = \psi'(a) $ and by concavity (cf. Lemma \ref{lem:critii}) one has for all $x \in (0, \frac{1}{2})$  
\begin{equation}
u(x) \leq u(a) + u'(a) (x-a) = \psi(a) + \psi'(a) (x-a) = \psi(x) ,
\end{equation}
a contradiction to the nonnegativity of $u$. Hence $u$ cannot touch $\psi$ on $(0, \frac{1}{2})$ and similarly one shows that $u$ cannot touch $\psi$ on $(\frac{1}{2},1)$. Morover, we assert that $u$ has to touch $\psi$, for if not then one obtains by Proposition \ref{prop:33} that $u \in C^\infty([0,1])$ and 
\begin{equation}
\frac{A_u'}{(1+u'^2)^\frac{5}{4}} \equiv \mathrm{const.} \quad \textrm{on $(0,1)$}.
\end{equation}
But since $A_u(0)= A_u(1) = 0 $ one can find a point $\xi \in (0,1)$ such that $A_u'(\xi) = 0$. Therefore 
\begin{equation}
\frac{A_u'}{(1+u'^2)^\frac{5}{4}} \equiv 0 
\end{equation}
and hence $A_u \equiv \mathrm{const}$. This yields $u'' \equiv 0$ but then the boundary conditions imply $u \equiv 0$, a contradiction to Assumption \ref{ref:ass1}. Hence $\{u = \psi \} = \{ \frac{1}{2} \}$. In particular by basic properties of the variational inequality there exists $C \in \mathbb{R}$ such that 
\begin{equation}\label{eq:elint}
A_u'(x) = C(1+ u'(x)^2)^\frac{5}{4}  \quad \forall x \in (0,\nicefrac{1}{2}).
\end{equation}  
As a further intermediate claim we assert that $C \neq 0 $. Indeed, if $C = 0$  then $A_u' \equiv 0$ which implies together with $A_u(0) = 0 $ that $A_u \equiv 0$. Then however $u'' \equiv 0$ on $(0, \frac{1}{2})$ which implies that $u' \equiv \mathrm{const}$ on $[0, \frac{1}{2}]$. But $u$ is symmetric and therefore $u'(\frac{1}{2}) =0$ resulting in $u' \equiv 0$. As a result $u\equiv 0$ which yields again a contradiction to $u \in C_\psi$. Hence $C \neq 0 $.

 As an indermediate claim, we assert that that $u$ is strictly concave on $(0, \frac{1}{2}]$, i.e. $u'' > 0 $ on $(0, \frac{1}{2}]$. To show this we can multiply \eqref{eq:elint} by $A_u$ and integrate to obtain
\begin{equation}
\frac{1}{2} A_u^2 =C ( u'(x) - u'(0) ) \quad \forall x \in (0, \nicefrac{1}{2}]
\end{equation}
and thus 
\begin{equation}\label{eq:dingsi}
u''(x)^2= 2C (u'(x)- u'(0)) (1+ u'(x)^2)^\frac{5}{2} \quad \forall x \in (0, \nicefrac{1}{2}]. 
\end{equation}
If there were now $x_0  \in (0 , \frac{1}{2}]$ such that $u''(x_0)= 0$, the above equation would imply that $u'(x_0) = u'(0)$ and because of monotonicity of $u'$ one has that $u' \equiv u'(0)$ on $(0,x_0)$. Another look at \eqref{eq:dingsi} implies that then $u'' \equiv 0 $ on $(0,x_0)$. This implies that $A_u' \equiv 0$ on $(0,x_0)$. This however is a contradiction to $C = 0$ when looking at \eqref{eq:elint}.
Since $u$ is strictly concave on $(0,\frac{1}{2}]$ we find that for all $x> 0$ one has $u'(x) < u'(0)$ and now
\begin{equation}
u''(x) =  - \sqrt{2|C|} \sqrt{u'(0)- u'(x) } ( 1+ u'(x)^2)^\frac{5}{4} \quad \forall x \in (0, \nicefrac{1}{2}).
\end{equation} 
Since the right hand side does not vanish on $(0, \frac{1}{2})$ we obtain 
\begin{equation}
\frac{- u''(x) }{\sqrt{u'(0)- u'(x) } ( 1+ u'(x)^2)^\frac{5}{4}} = \sqrt{2|C|} \quad \forall x \in (0, \nicefrac{1}{2}).
\end{equation}
Next we fix $\epsilon > 0$ and integrate from $\epsilon$ to some arbitary $x \in (0 , \frac{1}{2})$ to get after a substitution of $s = u'(x)$
\begin{equation}
\int_{u'(x)}^{u'(\epsilon)} \frac{1}{\sqrt{u'(0) - s } ( 1+ s^2)^\frac{5}{4}} \ds = \sqrt{2|C|} (x-\epsilon) 
\end{equation}
We can pass to the limit as $\epsilon \downarrow 0$ using the monotone convergence theorem on the left hand side to obtain
\begin{equation}
\int_{u'(x)}^{u'(0)} \frac{1}{\sqrt{u'(0) -s} (1+s^2)^\frac{5}{4}} \ds = \sqrt{2|C|} x.
\end{equation}
As $u$ is symmetric one has $u'(\frac{1}{2}) = 0  $ which implies 
\begin{equation}
\int_{0}^{u'(0)} \frac{1}{\sqrt{u'(0) -s} (1+s^2)^\frac{5}{4}} \ds = \frac{1}{2}\sqrt{2|C|} .
\end{equation}
Note that this means 
\begin{equation}
u'(x) = F^{-1}(x) \quad \forall x \in (0, \nicefrac{1}{2})
\end{equation}
where 
\begin{equation}
F(z) = \frac{\int_{z}^{u'(0)} \frac{1}{\sqrt{u'(0)- s} (1+s^2)^\frac{5}{4}} \ds}{2 \int_0^{u'(0)} \frac{1}{\sqrt{u'(0)-s} (1+s^2)^\frac{5}{4}}  \ds }.
\end{equation}
Therefore 
\begin{align}
u(x) & = \int_0^x u'(s) \ds = \int_0^x F^{-1}(s) \ds = \int_{F^{-1}(0)}^{F^{-1}(z) } z F'(z) \dz 
\\ & = \frac{1}{2} \frac{\int_{u'(x)}^{u'(0)} \frac{z}{\sqrt{u'(0) - z} {(1+z^2)^\frac{5}{4}}} \dz }{\int_{0}^{u'(0)} \frac{1}{\sqrt{u'(0) - z} {(1+z^2)^\frac{5}{4}}} \dz}.\label{eq:737}
\end{align}
Note that this already yields an equation for $u$ with only one free parameter, namely $u'(0)$. We show next that $u'(0)$ is uniquely determined by $\psi( \frac{1}{2})$.  To this end we compute
\begin{equation}\label{eq:eind}
\psi( \nicefrac{1}{2}) = u( \nicefrac{1}{2} ) = \frac{1}{2} \frac{\int_{0}^{u'(0)} \frac{z}{\sqrt{u'(0) - z} {(1+z^2)^\frac{5}{4}}} \dz }{\int_{0}^{u'(0)} \frac{1}{\sqrt{u'(0) - z} {(1+z^2)^\frac{5}{4}}} \dz} =: H(u'(0)),
\end{equation}
where 
\begin{equation}\label{eq:H}
H(A) := \frac{1}{2} \frac{\int_{0}^{A} \frac{z}{\sqrt{A - z} {(1+z^2)^\frac{5}{4}}} \dz }{\int_{0}^{A} \frac{1}{\sqrt{A - z} {(1+z^2)^\frac{5}{4}}} \dz}.
\end{equation}
We show in Appendix \ref{app:B} that $H$ is a strictly monotone and smooth function of $A$. Hence there exists one unique $A>0 $ such that $H(A) = \psi(\frac{1}{2})$. We conclude with \eqref{eq:eind} that $u'(0) = H^{-1}( \psi(\frac{1}{2}) ) $. Using \eqref{eq:737} again we find that $u$ satisfies 
\begin{equation}
u(x) = \frac{1}{2} \frac{\int_{u'(x)}^{H^{-1}( \psi(\frac{1}{2}) )} \frac{z}{\sqrt{H^{-1}( \psi(\frac{1}{2}) ) - z} {(1+z^2)^\frac{5}{4}}} \dz }{\int_{0}^{H^{-1}( \psi(\frac{1}{2}) )} \frac{1}{\sqrt{H^{-1}( \psi(\frac{1}{2}) ) - z} {(1+z^2)^\frac{5}{4}}} \dz}, \quad x \in (0, \frac{1}{2}) .
\end{equation}
Hence $u$ solves on $(0,\frac{1}{2})$ 
\begin{equation}\label{eq:diffeq}
\begin{cases}
u'(x) = J(u(x)) & x \in (0, \frac{1}{2})  \\
u(\frac{1}{2}) = \psi(\frac{1}{2}),
\end{cases}
\end{equation}
where $J$ is the inverse function to 
\begin{equation}
[0, H^{-1}(\psi(\nicefrac{1}{2}))] \ni
r \mapsto \frac{1}{2} \frac{\int_{r}^{H^{-1}( \psi(\frac{1}{2}) )} \frac{z}{\sqrt{H^{-1}( \psi(\frac{1}{2}) ) - z} {(1+z^2)^\frac{5}{4}}} \dz }{\int_{0}^{H^{-1}( \psi(\frac{1}{2}) )} \frac{1}{\sqrt{H^{-1}( \psi(\frac{1}{2}) ) - z} {(1+z^2)^\frac{5}{4}}} \dz},
\end{equation} 
which is well defined because of the positivity of the integrand. The only problem that remains is that maximal solutions to \eqref{eq:diffeq} are not necessarily unique as $J$ is not locally Lipschitz around $\psi(\frac{1}{2})$. It follows however by Lemma \ref{lem:75} that it does have a unique solution that is strictly concave in  $(0,\frac{1}{2}]$. As we have shown strict concavity of each critical point, $u$ is uniquely determined by 
\begin{equation}
\begin{cases}
u'(x) = J(u(x)) & x \in (0, \frac{1}{2}) , \\
u(\frac{1}{2}) = \psi(\frac{1}{2}), u(0) = 0 ,\\
u \; \textrm{strictly concave on $(0, \frac{1}{2}]$}. &  
\end{cases}
\end{equation}

Hence there can exist at most one such $u$ as in \eqref{eq:diffeq}. 
\end{proof}
\begin{proof}[Proof of Theorem \ref{thm:critical}]
Follows now immediately from the previous Lemma and Corollary  \ref{cor:minisym}. 
\end{proof}
\subsection{Compactness of the critical set}
In the rest of this section we discuss compactness of the set of critical points. This will be of high importance later when examining the convergence. Here we do not impose any further assumption on $\psi$ anymore, except for Assumption \ref{ref:ass1}.  
\begin{lemma}[Compactness of critical set]\label{lem:compact}
Suppose that $A < \frac{c_0^2}{4}$ and let 
\begin{equation}
M_{crit}(A) := \{ w \in C_\psi : D\E(w) (v-w ) \geq 0 \; \forall v \in C_\psi ,\;  \E(w) \leq A \}. 
\end{equation}
Then $M_{crit}(A)$ is compact in $W^{2,2}(0,1)$. 
\end{lemma}
\begin{proof}
Let $A$ be as in the statement. We show that $M_{crit}(A)$ is a bounded set in $W^{3,\infty}(0,1)$ and also closed in $W^{2,2}(0,1) \cap W_0^{1,2}(0,1)$. This immediately implies the compactness. For the boundedness in $W^{3,\infty}(0,1)$ first note that there exists some $\delta >0 $ such that $\psi < 0 $ on $[0, \delta] \cup [1-\delta,1]$. Now fix $w \in M_{crit}(A)$.  By Remark \ref{ref:remeng} one has $||w'||_\infty \leq G^{-1}(\sqrt{A})$ and  $||w''||_{L^2}^2 \leq A ( 1+ G^{-1}(\sqrt{A})^2)^\frac{5}{2}$. Similar to the derivation of \eqref{eq:73} we can conclude that for $A_w := \frac{w''}{(1+ w'^2)^\frac{5}{4}}$
\begin{equation}
2\int_0^1 \frac{A_w' \phi'}{(1+ w'^2)^\frac{5}{4}} \leq 0 \quad \forall \phi \in C_0^\infty(0,1) .
 \end{equation}
 By \cite[Lemma 37.2]{Tartar} there exists a Radon measure $\mu$ on $(0,1)$ which is by \eqref{eq:72} supported on $\{ u = \psi \}$ such that 
 \begin{equation}\label{eq:746}
 2\int_0^1 \frac{A_w' \phi'}{(1+ w'^2)^\frac{5}{4}} = \int_0^1 \phi \;  d\mu \quad \forall \phi \in C_0^\infty(0,1).
 \end{equation}
 By \eqref{eq:eulernaiv} we also find 
 \begin{equation}\label{eq:naiveulermass}
 2 \int_0^1 \frac{w'' \phi''}{ (1+w'^2)^\frac{5}{2} } - 5 \int_0^1 \frac{w''^2 w' \phi'}{(1+ w'^2)^\frac{7}{2}} = \int_0^1 \phi \; d\mu  \quad \forall \phi \in C_0^\infty(0,1).  
 \end{equation}
 Note that since $w$ is nonnegative by Remark \ref{rem:posi} one has $\{ u= \psi \} \subset [\delta, 1- \delta]$ and hence $\mu$ is finite.  Moreover one can plug into \eqref{eq:naiveulermass} a function $\phi \in C_0^\infty(0,1)$ such that $\phi = 1 $ on $[\delta, 1-\delta]$, $0 \leq \phi \leq 1$ and $|| \phi' ||_\infty < \frac{2}{\delta} $ as well as $||\phi''||_\infty < \frac{2}{\delta^2}$ to find
 \begin{equation}\label{eq:mu01}
 \mu((0,1)) \leq \frac{10}{\delta} \int_0^1 \frac{w''^2|w'|}{(1+w'^2)^\frac{7}{2}}  + \frac{4}{\delta^2} \int_0^1 \frac{|w''|}{(1+ w'^2)^\frac{5}{2}} \leq \frac{10A}{\delta} + \frac{4\sqrt{A}}{\delta^2}. 
 \end{equation}
 Going back to \eqref{eq:746} we obtain 
 \begin{equation}
 \int_0^1  \left( \frac{2A_w' }{(1+ w'^2)^\frac{5}{4}}- m \right) \phi'  = 0 \quad \forall \phi \in C_0^\infty(0,1),
 \end{equation}
 where $m(t) := \mu((t,1))$ is a function bounded by $\mu(0,1)$. We conclude
 \begin{equation}
 \frac{2A_w' }{(1+ w'^2)^\frac{5}{4}}- m  \equiv \mathrm{const.} = \int_0^1\left(\frac{2A_w' }{(1+ w'^2)^\frac{5}{4}}- m \right) .
\end{equation}  
This implies 
\begin{equation}\label{eq:Awp}
|A_w'| \leq \frac{1}{2} (1 + G^{-1}(\sqrt{A})^2)^\frac{5}{4} \left ( 2 \mu(0,1) + \left\vert \int_0^1 \frac{2A_w'}{(1+w'^2)^\frac{5}{4}} \right\vert \right).  
\end{equation}
By Lemma \ref{lem:critii} we can integrate by parts without boundary terms and obtain
\begin{equation}
\left\vert \int_0^1 \frac{A_w'}{(1+w'^2)^\frac{5}{4}} \right\vert = \left\vert - \frac{5}{2}\int_0^1 A_w^2 \frac{w'}{(1+ w'^2)} \right\vert \leq \frac{5}{2}\E(w) \leq\frac{5}{2} A.
\end{equation}
Together with this \eqref{eq:Awp} and \eqref{eq:mu01} we obtain 
\begin{equation}\label{eq:Awpp}
|A_w'| \leq \frac{1}{2} (1 + G^{-1}(\sqrt{A})^2)^\frac{5}{4} \left ( 2\left( \frac{10A}{\delta} + \frac{4\sqrt{A}}{\delta^2} \right)  + 5A \right) =: D(A,\delta).
\end{equation}
Note that $\delta$ is chosen independently of $w$. This also implies that for all $x \in (0,1)$ one has
\begin{equation}\label{eq:aq}
|A_w(x)| \leq |A_w(0)| + \int_0^x|A_w'(s)| ds \leq D(A,\delta)
\end{equation} 
as $A_w(0) = w''(0) = 0$. Since $||w'||_\infty < G^{-1}(\sqrt{A})$ we obtain with the explicit formula for $A_w$ that 
\begin{equation}
|w''| \leq   (1 + G^{-1}(\sqrt{A})^2)^\frac{5}{4} D(A,\delta) .
\end{equation}
Finally note that 
\begin{equation}
A_w' = \frac{w'''}{(1+ w'^2)^\frac{5}{4}} - \frac{5}{2} A_w^2 \frac{w'}{(1+ w'^2)}.
\end{equation}
Combining this with \eqref{eq:Awpp} and \eqref{eq:aq} we get 
\begin{equation}
|w'''| \leq (1 + G^{-1}(\sqrt{A})^2)^\frac{5}{4} ( D(A,\delta) + \frac{5}{4} D(A,\delta)^2 ). 
\end{equation}
We have bounded $||w''||_\infty $ and $||w'''||_\infty$ with bounds that depend only on $A$ and $\psi$. This implies that $M_{crit}(A)$ is bounded in $W^{3,\infty}(0,1)$. This makes it precompact in $W^{2,2}(0,1)$. The closedness of $M_{crit}(A)$ in  $W^{2,2}(0,1)\cap W_0^{1,2}(0,1)$ follows by an easy computation using \eqref{eq:eulernaiv}.  
\end{proof}
\section{Convergence Behavior}
In this section we want to examine whether the flow converges in the energy space $W^{2,2}(0,1)$. For large obstacles the absence of critical points already shows  nonconvergence, cf. \cite[Corollary 5.22]{Marius2}. For small obstacles and small initial energies however, convergence is true.

\begin{lemma} [$W^{2,2}$-subconvergence disregarding small sets]\label{lem:epsex} 
Let $u_0 \in C_\psi$ be such that $\E(u_0) < \frac{c_0^2}{4}$. Then for each $\epsilon >0$ there exists a set $B \subset [0,\infty)$ with $|B|< \epsilon$ such that $u_{\mid_{[0,\infty) \setminus B}}: {[0,\infty) \setminus B} \rightarrow W^{2,2}(0,1)$ is $W^{2,2}(0,1)$-subconvergent (in the sense of Definition \ref{def:210}) to points in $M_{crit}$, where $M_{crit}$ is defined as in \eqref{eq:mcrit}.
\end{lemma}

\begin{proof}
Let $\epsilon > 0 $ and let $\phi$ be a nonincreasing function that coincides with $\mathcal{E} \circ u$ almost everywhere. Let $M$ be as in the statement. Define for each $n \in \mathbb{N}$ the set 
\begin{equation}
Q_n := \left\lbrace t  > 0 : \mathcal{E}(u(t)) \neq \phi(t) \; \textrm{or FVI is not true at $t$ or} \;  ||\dot{u}(t)||_{L^2} > \frac{1}{n} \right\rbrace. 
\end{equation}
Note that by Chebyshov's inequality, $Q_n$ has finite measure for all $n$ and therefore there exists some $k_n \in \mathbb{N}$ such that $|Q_n \cap [k_n , \infty)| < \frac{\epsilon}{2^n}$. Without loss of generality we can also achieve that $k_{n+ 1} > k_n$ for all $n \in \mathbb{N}$.  Therefore $| \bigcup_{ n = 1}^\infty (  Q_n \cap [k_n , \infty))  | < \epsilon$. Define $B := \bigcup_{ n = 1}^\infty ( Q_n \cap [k_n , \infty)) $. Suppose now that $(\theta_n)_{n = 1}^\infty$ is an arbitrary sequence satisfying $(\theta_n) \subset [0,\infty) \setminus B$ and $\theta_n \rightarrow \infty$. 

By Remark \ref{ref:remeng} $(u(\theta_n))_{n= 1}^\infty$ is bounded in $W^{2,2}(0,1)$ and hence there exists a 
subsequence $(\theta_{l_n})_{n = 1}^\infty$ and $u_\infty \in C_\psi$ such that $ \theta_{l_n} > k_n $ for all $n \in \mathbb{N}$, $u(\theta_{l_n}) \rightarrow u_\infty $ in $W^{1,\infty}(0,1)$ and $u(\theta_{l_n}) \rightharpoonup u_\infty$ in $W^{2,2}(0,1)$. It remains to show that this convergence is strong in $W^{2,2}(0,1)$ and $u_\infty \in M_{crit}$. Note that $\theta_{l_n} > k_n$ implies that $\theta_{l_n} \in Q_n^C$ for all $n \in \mathbb{N}$. We verify that $(u(\theta_{l_n}) )_{n = 1}^\infty$ is a Cauchy sequence in $W^{2,2}(0,1) \cap W_0^{1,2}(0,1)$. Using the definition of $Q_n$, the $FVI$ equation and again the Lipschitz continuity of $x \mapsto \frac{1}{(1+x^2)^\frac{5}{2}}$
we can compute
\begin{align*}
& \frac{1}{( 1 + G^{-1} (\sqrt{\E(u_0)})^2) ^\frac{5}{2}}  \int_0^1 ( u(\theta_{l_n})'' - u(\theta_{l_m}) )'')^2 \dx  \leq \int_0^1 \frac{(u(\theta_{l_n})'' - u(\theta_{l_m})'')^2}{(1+ u(\theta_{l_m})'^2)^\frac{5}{2}} \dx
\\ & = \int_0^1 \frac{u(\theta_{l_n})''^2}{(1+ u(\theta_{l_m})'^2)^\frac{5}{2}} \dx- \int_0^1 \frac{u(\theta_{l_m})''^2}{(1+ u(\theta_{l_m})'^2)^\frac{5}{2}}  \dx  \\ & \quad +   2 \int_0^1 \frac{u(\theta_{l_m})'' ( u(\theta_{l_m})''- u(\theta_{l_n})'')}{(1 + u(\theta_{l_m})'^2)^\frac{5}{2}} \dx
\\ & = \int_0^1 \frac{u(\theta_{l_n})''^2}{(1+u(\theta_{l_n})'^2)^\frac{5}{2}} \dx + \int_0^1 u(\theta_{l_n})''^2 \left[ \frac{1}{(1 + u(\theta_{l_m})'^2)^\frac{5}{2}} - \frac{1}{(1 + u(\theta_{l_n})'^2)^\frac{5}{2}}  \right] \\ & \quad  -  \int_0^1 \frac{u(\theta_{l_m})''^2}{(1+ u(\theta_{l_m})'^2)^\frac{5}{2}} \dx +  D\E(u(\theta_{l_m}))(u(\theta_{l_m}) - u(\theta_{l_n}))  \\ & \quad + 5 \int_0^1 \frac{u(\theta_{l_m})''^2 u(\theta_{l_m})'}{(1 +u(\theta_{l_m})'^2)^\frac{7}{2}} (u(\theta_{l_m})' - u(\theta_{l_n})' ) \dx
\\ & \leq \E(u(\theta_{l_n})) - \E(u(\theta_{l_m}))   \\ & \quad  + \frac{5}{2} \E(u_0) ( 1 + G^{-1}(\sqrt{\E(u_0)}) ^2)^\frac{5}{2} ||u(\theta_{l_m})' - u(\theta_{l_n})'||_{L^\infty} \\ &   \quad + ( \dot{u}(\theta_{l_m}), u(\theta_{l_n}) - u(\theta_{l_m}) ) + 5 \E(u(\theta_{l_m})) || u(\theta_{l_m})' - u(\theta_{l_n})'||_{L^\infty}
\\ & = \phi(\theta_{l_n}) - \phi(\theta_{l_m}) + (D+ || \dot{u}(\theta_{l_m}) ||_{L^2}) || u(\theta_{l_m})' - u(\theta_{l_n})'||_{L^\infty}
\\ & \leq \phi(\theta_{l_n}) - \phi(\theta_{l_m}) + (D+ 1) || u(\theta_{l_m})' - u(\theta_{l_n})'||_{L^\infty},
\end{align*}
for some fixed constant $D> 0$. 
Since $\phi$ is nonincreasing, $(\phi(\theta_{l_n}))_{n = 1}^\infty$ is a Cauchy sequence and hence the Cauchy property of $(u(\theta_{l_n}) )_{n = 1}^\infty \subset W^{2,2}(0,1)$ is shown. Hence $u(\theta_{l_n}) \rightarrow u_\infty$ in $W^{2,2}(0,1)$. Moreover 
\begin{equation}
(\dot{u}(\theta_{l_n}), v - u(\theta_{l_n}))_{L^2} + D\E(u(\theta_{l_n}))(v- u(\theta_{l_n})) \geq 0 .
\end{equation}
The fact that $||\dot{u}(\theta_{l_n})||_{L^2} < \frac{1}{n}$ since $\theta_{l_n} \in Q_n^C$ implies that 
\begin{equation}\label{eq:prime0}
(\dot{u}(\theta_{l_n}), v - u(\theta_{l_n}))_{L^2} \rightarrow 0. 
\end{equation}
Here we also used that $u \in L^\infty((0, \infty), W^{2,2}(0,1)) \subset L^\infty((0, \infty), L^2(0,1))$. By a direct computation that uses the just derived $W^{2,2}$-convergence it is also easy to see that 
\begin{equation}
 D\E(u(\theta_{l_n}))(v- u(\theta_{l_n}))  \rightarrow D\E(u_\infty) (v- u_ \infty). 
\end{equation}
We conclude with this and \eqref{eq:prime0} that $u_\infty \in M_{crit}.$  
\end{proof}
So far we have proved a $W^{2,2}$-subconvergence result for $FVI$ gradient flows with an exceptional set $B$ of artbitrary small measure. The next step is now to use the uniform Hölder continuity of $FVI$ gradient flows in $L^2(0,1)$ to get rid of the exceptional set. The topology however changes for the worse but can be improved again in the rest of the section. 
 \begin{lemma}[Full $L^2$-subconvegence to critical points]\label{lem:fullsub}
 Let $u_0 \in C_\psi$ be such that $\E(u_0) < \frac{c_0^2}{4}$ and let $u$ be an FVI gradient flow  starting at $u_0$. Let $M_{crit}$ be as in \eqref{eq:mcrit}. Then $u : [0,\infty) \rightarrow L^2(0,1)$ is fully $L^2$-subconvergent to points in $M_{crit}$. 
 \end{lemma}
\begin{proof}
Let $u_0$,$u$ be as in the statement. We start by showing full $L^2(0,1)$-subconvergence. For this let $t_n \rightarrow \infty$. Now set $\epsilon_1 := 1 > 0 $ and apply Lemma \ref{lem:epsex} with $\epsilon = \epsilon_1$. This yields a set $B= B(\epsilon_1)$ such that $|B(\epsilon_1)| <\epsilon_1$ and  $u_{\mid_{[0, \infty) \setminus B(\epsilon_1)}} $ $W^{2,2}$-subconverges to points in $M_{crit}$. Note that for all $n \in \mathbb{N}$ there exists $s_n^1 \in (t_n - \epsilon_1 , t_n + \epsilon_1)$ such that $s_n^1 \in [0,\infty) \setminus B(\epsilon_1)$, since the contrapositive of this statement contradicts $|B(\epsilon_1)|< \epsilon_1$. By Lemma \ref{lem:epsex} there exists a subsequence $(l_n^1)_{n = 1}^\infty \subset \mathbb{N}$ such that  $u(s_{l_n^1}^1)$ converges in $W^{2,2}$ to some $u_\infty^1 \in M_{crit}$. Note that by Proposition \ref{prop:holder} one there exists $D> 0$ such that 
\begin{align}
\limsup_{n \rightarrow \infty} &||u(t_{l_n^1}) - u_\infty^1 ||_{L^2(0,1)} \leq  \limsup_{n \rightarrow \infty} (||u(t_{l_n^1})- u(s_{l_n^1}^1)||_{L^2} + ||u(s_{l_n^1}^1) - u_\infty^1 ||_{L^2})
\\ & =  \limsup_{n \rightarrow \infty} ||u(t_{l_n^1})- u(s_{l_n^1}^1)||_{L^2}
 \leq \limsup_{n \rightarrow \infty} D \sqrt{|t_{l_n^1}- s_{l_n^1}^1 | } \leq D \sqrt{\epsilon_1}.
\end{align} 
Hence there exists some $n_1 \in \mathbb{N}$ such that for all $n \geq n_1$ 
\begin{equation}
||u(t_{l_n^1}) - u_\infty^1 ||_{L^2(0,1)}  \leq 2D\sqrt{\epsilon_1}.
\end{equation}
We start an iterartive procedure by repeating the process starting with the sequence $(t_{l_n^1})_{n \geq n_1}$ and for $\epsilon_2 := \frac{\epsilon_1}{2}$, more precisely: We again choose a measurable set $B(\epsilon_2)$ of measure smaller than $\epsilon_2$ such that $u_{[0,\infty) \setminus B(\epsilon_2) }$ $W^{2,2}$-subconverges to points in $M_{crit}$. We again observe that for all $n >n_1$ there exists $s_n^2 \in (t_{l_n^1} - \epsilon_2 , t_{l_n^1}  + \epsilon_2 ) \cap [0,\infty) \setminus B(\epsilon_2)$. Therefore we can find a subsequence of $(s_n^2)_{n \geq n_1}$ along which $u$ converges to some $u_\infty^2 \in M_{crit}$. As above this yields now a subsequence $(t_{l_n^2})_{n \geq n_1}$ of $(t_{l_n^1})_{n \geq n_1}$ such that 
\begin{equation}
\limsup_{n \rightarrow \infty} ||u(t_{l_n^2}) - u_\infty^2 ||_{L^2(0,1)} < D \sqrt{\epsilon_2}
\end{equation}
In particular we can choose $n_2 \geq n_1$ such that for all $n \geq n_2$ 
 \begin{equation}
 ||u(t_{l_n^2}) - u_\infty^2 ||_{L^2(0,1)} < 2 D \sqrt{\epsilon_2}.
 \end{equation}
 Keeping going, we can find for all $k \in \mathbb{N}$ nested subsequences 
 $(t_{l_n^k})_{n > n_k } \subset (t_{l_n^{k-1}})_{n \geq n_{k-1} } \subset ... \subset (t_n)_{n \geq 1}$ and $ \{u_\infty^1, ... , u_{\infty}^k \} \subset M_{crit}$ such that for all $q \in \{1,...,k\}$ and $n > n_q$ one has
  \begin{equation}\label{eq:8212}
 ||u(t_{l_n^q}) - u_\infty^q ||_{L^2(0,1)} < 2 D \sqrt{\epsilon_q} = 2D \frac{1}{\sqrt{2}^q} \sqrt{\epsilon_1}.
 \end{equation}
Because of the compactness of $M_{crit}$ by Lemma \ref{lem:compact} we obtain that $(u_\infty^{q})_{q=1}^\infty$ has a $W^{2,2}$-covergent subsequence, denoted by $(u_\infty^{q_m})_{m = 1}^\infty$. We denote the limit of this sequence simply by $u_\infty \in M_{crit}$. 
 The subsequence we consider now is $(t_{l_{n_{q_m}}^{q_m}})_{m \in \mathbb{N}}$. For the sake of simplicity of notation we define $a_m := t_{l_{n_{q_m}}^{q_m}}$. Now we observe by \eqref{eq:8212} 
 \begin{equation*}
 ||u(a_m) - u_\infty||_{L^2} \leq || u(a_m) - u_\infty^{q_m}||_{L^2} + || u_\infty^{q_m} - u_\infty||_{L^2} \leq 2D  \sqrt{\epsilon_1} \frac{1}{\sqrt{2}^{q_m}} + || u_\infty^{q_m} - u_\infty||_{L^2}.
 \end{equation*}
 Now both terms on the right hand side of this inequality tend to zero as $m \rightarrow \infty$ and thus $u(a_m) \rightarrow u_\infty$ in $L^2(0,1)$. As $(a_m)_{m=1}^\infty$ was a subsequence of $(t_n)_{n = 1}^\infty$ the claim follows.  
\end{proof}

\begin{proof}[Proof of Theorem \ref{thm:convergence}]
Let $(t_n)_{n = 1}^\infty$ be an arbitrary sequence that diverges to infinity. By Lemma \ref{lem:fullsub} there exists a subsequence which we call again  $t_{n}$ and some $u_\infty \in M_{crit}$ such that $u(t_{n}) \rightarrow u_\infty$ in $L^2(0,1)$. By Theorem \ref{thm:existence}, $u(t_{n})$ is bounded in $W^{2,2}(0,1)$ hence we can choose a further subsequence which we do not relabel such that $u(t_n) \rightharpoonup u_\infty$ weakly in $W^{2,2}(0,1)$. By compact embedding we obtain $u(t_n) \rightarrow u_\infty$ in $C^1([0,1])$ and hence the claim follows. 
\end{proof}
\begin{proof}[Proof of Theorem \ref{thm:Convconve}]
Let $u_0$,$u$ be as in the statement. By Corollary \ref{cor:sympres} one has $u(t) ( 1- \cdot) = u(t)$ for all $t>0$. Let now $w \in C_\psi$ be the unique symmetric critical point in $C_\psi$ (cf. Theorem \ref{thm:critical}). By Theorem \ref{thm:critical}, $w$ is a minimizer of $\E$ in $C_\psi$. Now let $t_n \rightarrow \infty$ be a sequence. Observe that by Theorem \ref{thm:convergence} there exists a subsequence $t_{l_n} \rightarrow \infty$ such that $u(t_{l_n})$ converges in $W^{1,\infty}(0,1)$ to some critical point $u_\infty \in C_\psi$. Now since $u(t_{l_n})(1- \cdot) = u(t_{l_n}) $ for all $n \in \mathbb{N}$ one obtains by the $L^2$-convergence that $u_\infty( 1- \cdot) = u_\infty$. From this follows that $u_\infty = w$ by Theorem \ref{thm:critical}. By the Urysohn property of $L^2$-convergence we obtain that $u(t_n) \rightarrow w$ in $L^2(0,1)$. As the sequence $(t_n)_{n = 1}^\infty$ was arbitrary we obtain that $u(t) \rightarrow w$ as $t \rightarrow \infty$. 
\end{proof}


\section{Open problems and perspectives}

In this final section we summarize some problems that could be interesting for future research. We also discuss some ways to approach them.

\begin{problem}[Optimal energy dissipation]
The article shows that energy-dissipating $FVI$ gradient flows can also be constructed even if the energy $\mathcal{E}$ is not $L^2$-semiconvex. It is however unclear whether the energy dissipation rate is optimal. For $L^2$-semiconvex functionals the dissipation rate will be optimal --- in the sense of $EDI$-gradient flows in optimal transport theory, cf. \cite[Definition 4.3]{usersguide}. Even more general --- if $\mathcal{E}$ can be written as a sum of a convex and a Frechét differentiable functional then each $FVI$ gradient flow is an $EDI$-gradient flow. This can be shown following the lines of \cite[Section 2.3]{Marius2}. It is vital for the theory to understand what role the convexity assumption plays for the energy dissipation.
\end{problem}

\begin{problem}[Energy threshold and geometry]
We have shown existence of the flow only below the energy threshold $\mathcal{E}(u_0) < \frac{c_0^2}{4}$. The reason for this threshold is that below one can obtain uniform control of $||\partial_xu(t,\cdot)||_{L^\infty}$ and hence one has control of the nonlinearities. While this is helpful for our analysis, the control is lost for large obstacles, cf. \cite{Marius1}, \cite[Section 5]{Marius2}. The reason is that $||\partial_x u||_{L^\infty}$ is a quantity that disregards the nature of $\mathcal{E}$ as a \emph{geometric energy of curves}, namely
\begin{equation}
\mathcal{E}(u) = \int_{\mathrm{graph(u)}} \kappa^2 \; \mathrm{d} \mathbf{s}  . 
\end{equation}
More precisely: If $||\partial_x u||_{L^\infty}$ becomes large, $\mathrm{graph}(u)$ is not necessarily ill-behaved as a curve. If one wants to go beyond the threshold of $\frac{c_0^2}{4}$ one needs to work with curves and formulate a geometric minimizing movement scheme. While this causes additional difficulties, there has recently been progress, eg. in \cite{Blatt}, for gradient flows of the ($p$-)elastic energy without an obstacle constraint.   
\end{problem}

\begin{problem}[Symmetry breaking or not?]
In the article we have seen that symmetric evolutions with $\mathcal{E}(u_0) < \frac{c_0^2}{4}$ approach the unique symmetric critical point from Lemma \ref{lem:symcone}. (Note that we have only shown uniqueness of this critical point, but its existence follows from symmetry-preserving and subconvergence --- or alternatively from \cite{Yoshizawa2}). 

We actually want to show convergence to a global minimizer. For this we have to show that a symmetric minimizer can be found. 
We have done so in Corollary \ref{cor:minisym} --- but again only below an energy threshold of $G(2)^2$, which is even smaller than $\frac{c_0^2}{4}$. 

The value  of $G(2)^2$ corresponds to a loss of convexity of $\frac{1}{G^{-1}}$ and hence poses a limitation to the nonlinear Talenti symmetrization. We expect that there exist symmetric minimizers also above this threshold, but a proof will require further techniques. Presumably one needs to find a more geometric approach to the symmetry problem, which will be subject to our future research. 
\end{problem}

\appendix
\section{Technical Proofs in Section \ref{sec:exi}}  \label{app:A}
\begin{proof}[Proof of Lemma \ref{lem:discr}]
Note first that 
\begin{equation}
I := \inf_{u \in C_\psi} \Phi_\tau^f(u) \leq \Phi_\tau^f(f) = \E(f) < \frac{c_0^2}{4}.
\end{equation}
Therefore we can choose a minimizing sequence $(u_n)_{n = 1}^\infty \subset C_\psi$ such that $\Phi_\tau^f(u_n) \leq \E(f)$ for all $n \in \mathbb{N}$. Hence 
\begin{equation}
\E(u_n) \leq \Phi_\tau^f(u_n) \leq \E(f) \quad \forall n \in \mathbb{N} 
\end{equation}
 By Remark \ref{ref:remeng} we obtain that $||u_n'||_{L^\infty} \leq G^{-1} ( \sqrt{\E(f)} ) $ and $(u_n)_{n = 1}^\infty$ is uniformly bounded in $W^{2,2}(0,1)$.
After choosing an appropriate subsequence (which we do not relabel), we can assume that $u_n \rightharpoonup w$ in $W^{2,2}(0,1)$ for some $w \in W^{2,2}(0,1)$. Note that $w \in C_\psi$ since $C_\psi$ is weakly closed as closed convex subset of $W^{2,2}(0,1)$. By Sobolev embedding $u_n \rightarrow w$ in $C^1([0,1])$ and hence in particular in $L^2(0,1)$. Just like in the proof of \cite[Lemma 2.5]{Anna} we obtain now that 
\begin{equation}\label{eq:eee}
\E(w) \leq \liminf_{n \rightarrow \infty} \E(u_n). 
\end{equation}
Because of the $L^2$-convergence we get 
\begin{equation}\label{eq:tauu}
\frac{1}{2\tau} || w - f||_{L^2}^2 = \lim_{n \rightarrow \infty} \frac{1}{2\tau} || u_n - f||_{L^2}^2 .
\end{equation}
Summing \eqref{eq:eee} and \eqref{eq:tauu} we obtain 
\begin{equation}
\Phi_\tau^f(w) \leq \liminf_{n \rightarrow \infty} \left( \E(u_n)+ \frac{1}{2\tau} || u_n - f||_{L^2}^2 \right) = I .
\end{equation}
Since $w \in C_\psi$ is admissible we also have $\Phi_\tau^f(w) \geq I$, which implies that $w$ is a minimizer. Equation \ref{eq:verreg} now follows easily from the fact that for all $v  \in C_\psi$ one has 
\begin{equation}
0 \leq \frac{d}{dt}_{\mid_{t = 0 }} \phi_\tau^f(w+ t(v-w)) 
\end{equation}
which is due to the fact that $w + t(v-w)$ is admissible for all $t \in [0,1]$. 
\end{proof}
\begin{proof}[Proof of Lemma \ref{lem:miii}]
Remark \ref{ref:remeng}  and Remark \ref{rem:minibound} yield that for each $k , \tau $ one has $||u_{k \tau}'||_\infty \leq G^{-1} ( \sqrt{\E(u_0)}) < \infty$. Let $C_E$ be the operator norm of the inclusion operator $W^{2,2}(0,1)\cap W_0^{1,2}(0,1) \hookrightarrow W^{2,2}(0,1)$. Then
\begin{align}
    ||u_{k \tau }||_{W^{2,2}}^2 & \leq C_E || u_{k \tau }||_{W^{2,2} \cap W_0^{1,2}}^2 = C_E \int_0^1 (u_{k \tau })''^2 \dx \nonumber \\ & \leq C_E (1+ G^{-1}(\sqrt{\E(u_0)})^2)^\frac{5}{2} \mathcal{E}(u_{k\tau})   \leq C_E (1+ G^{-1}(\sqrt{\E(u_0)})^2)^\frac{5}{2} \mathcal{E}(u_0) \label{eq:416}
\end{align}
Now 
\begin{align*}
    ||u^\tau(t)||_{W^{2,2}} & \leq \frac{(k+1)\tau - t}{\tau }||u_{k \tau}||_{W^{2,2}} + \frac{t - k\tau }{\tau }|| u_{k \tau }||_{W^{2,2}}
   \\ &  \leq ||u_{k \tau }||_{W^{2,2}} + || u_{(k + 1) \tau }||_{W^{2,2}} \leq 2 \sqrt{C_E (1+ G^{-1}(\sqrt{\E(u_0)})^2)^\frac{5}{2} \mathcal{E}(u_0) }.
\end{align*}
\end{proof}
 \begin{proof}[Proof of Lemma \ref{lem:w3infbds}]
  Fix $t, \tau > 0 $. For the sake of simplicity of notation we define $u := \overline{u}^\tau(t)$. 
  First we expand \eqref{eq:VARDISCR} to find that for each $\phi \in W^{2,2}(0,1) \cap W_0^{1,2}(0,1)$  such that $\phi \geq 0 $
  \begin{equation}
  \int_0^1 \dot{u}^\tau (t) \phi \dx  + 2 \int_0^1 \frac{u'' \phi'' }{(1+ u'^2)^\frac{5}{2}} \dx - 5 \int_0^1 \frac{u''^2 u' \phi'}{(1+ u'^2)^\frac{7}{2}} \dx \geq 0  
  \end{equation}
and for each $\phi \in W^{2,2}(0,1) \cap W_0^{1,2}(0,1)$ supported on $\{ u > \psi \}$  one has 
   \begin{equation}\label{eq:DINGSI}
  \int_0^1 \dot{u}^\tau (t) \phi \dx  + 2 \int_0^1 \frac{u'' \phi'' }{(1+ u'^2)^\frac{5}{2}} \dx - 5 \int_0^1 \frac{u''^2 u' \phi'}{(1+ u'^2)^\frac{7}{2}} \dx = 0.  
  \end{equation}
By a version of the Riesz-Markow-Kakutani Theorem (see \cite[Lemma 37.2]{Tartar}), there exists a Radon measure $\mu$ on $(0,1)$ such that for each $\phi \in C_0^\infty(0,1)$
 \begin{equation}\label{eq:RMK}
  \int_0^1 \dot{u}^\tau (t) \phi \dx  + 2 \int_0^1 \frac{u'' \phi'' }{(1+ u'^2)^\frac{5}{2}} \dx - 5 \int_0^1 \frac{u''^2 u' \phi'}{(1+ u'^2)^\frac{7}{2}} \dx=  \int_{(0,1)} \phi \; \mathrm{d}\mu. 
  \end{equation}
  Equation \eqref{eq:DINGSI} implies that $\mu$ is supported on $\{ u = \psi\}$. 
  Because of the assumptions on the obstacle $\mu$ is a Radon measure with support compactly contained in $(0,1)$, hence also a finite measure. Now we want to bound $\mu((0,1))$ independently of $\tau$. As an intermediate claim we assert that there exists $\delta > 0 $ independent of $\tau$ such that $u >\psi $ on $[0, \delta] \cup [1- \delta , 1]$. For this note that there exists $\delta_1 > 0 $ such that $\psi <  \frac{1}{2}\min\{ \psi(0), \psi(1) \} := P<0 $ on $[0,\delta_1]\cup [1- \delta_1 , 1]$. By uniform boundedness of $u^\tau$ in $L^\infty((0,\infty), W^{2,2}(0,1))$ (cf. Lemma \ref{lem:miii}) there exists a universal constant $L$ independent of $\tau$ such that $||u'||_\infty < L$. This and $u(0) = u(1) =0$ imply that $u \geq P$ on $[0, \frac{|P|}{L}] \cup [ 1- \frac{|P|}{L}, 1]$. Choosing $\delta:= \min \{ \delta_1 , \frac{|P|}{L} \}$ we obtain the intermediate claim.  
   We can now plug into \eqref{eq:RMK} a function $\phi \in C_0^\infty(0,1)$ such that $\phi \equiv 1 $ on $[\delta , 1- \delta ] $, $0 \leq \phi \leq 1$ and $||\phi'||_\infty < \frac{2}{\delta}, ||\phi''||_\infty< \frac{2}{\delta^2} $. This yields 
  \begin{align}
  \quad \quad \mu(&(0,1))  = \int_0^1 \dot{u}^\tau (t) \phi \dx  + 2 \int_0^1 \frac{u'' \phi'' }{(1+ u'^2)^\frac{5}{2}} \dx - 5 \int_0^1 \frac{u''^2 u' \phi'}{(1+ u'^2)^\frac{7}{2}} \dx \\ & \leq || \dot{u}^\tau (t) ||_{L^2}  + 2||\phi''||_\infty\int_0^1 \frac{|u''|}{(1+u'^2)^\frac{5}{2}} \dx   + 5 || \phi'||_\infty \int_0^1 \frac{u''^2}{(1+ u'^2)^\frac{5}{2}} \dx
  \\ &  \leq   || \dot{u}^\tau (t) ||_{L^2}  + 2||\phi''||_\infty \left(\int_0^1 \frac{u''^2}{(1+u'^2)^\frac{5}{2}} \dx \right)^\frac{1}{2}  + 5 || \phi'||_\infty \int_0^1 \frac{u''^2}{(1+ u'^2)^\frac{5}{2}} \dx 
  \\ & \leq || \dot{u}^\tau (t) ||_{L^2}  + \frac{4 \sqrt{\E(u_0)}}{\delta^2} + \frac{10 \mathcal{E}(u_0)}{\delta}   =: ||\dot{u}^\tau(t)|| + A \label{eq:utau},
  \end{align}
 for some $A = A(\delta, u_0)$. Now observe that 
  \begin{equation}
  \int_0^1 \phi \; d\mu = \int_0^1 \int_0^x \phi'(y) \dy  \; d\mu = \int_0^1 \mu((y,1) ) \phi'(y) \dy  .
  \end{equation}
  Defining $m(t) := \mu((t,1))$ we obtain by \eqref{eq:RMK} that for each $\phi \in C_0^\infty(0,1)$  
  \begin{equation}\label{eq:430}
  \int_0^1 \dot{u}^\tau (t) \phi \dx  + 2 \int_0^1 \frac{u'' \phi'' }{(1+ u'^2)^\frac{5}{2}} - 5 \int_0^1 \frac{u''^2 u' \phi'}{(1+ u'^2)^\frac{7}{2}} \dx=  \int_0^1 m \phi' \dx. 
  \end{equation}
  Now fix $\theta \in C_0^\infty(0,1)$ such that $\int_0^1 \theta = 1$ and let $\eta \in C_0^\infty(0,1)$ be arbitrary. Observe that $\phi(x) := \int_0^x \eta (r) \dr - \int_0^x\theta(r)  \dr \int_0^1 \eta(y) \dy$ lies in $C_0^\infty(0,1)$. Plugging this in \eqref{eq:430} we infer 
  \begin{align}
  2 \int_0^1 \frac{u'' \eta'}{(1+ u'^2)^\frac{5}{2}}& = - \int_0^1 \int_0^x \dot{u}^{\tau}(t)(x) \eta(r) \dr \dx + \int_0^1 \int_0^x \dot{u}^{\tau}(t)(x) \theta(s) \ds \dx \int_0^1 \eta(r) \dr \\ & \quad  + 2 \int_0^1 \frac{u'' \theta '}{(1+ u'^2)^\frac{5}{2}} \dx \int_0^1 \eta(r) \dr + 5 \int_0^1 \frac{u''^2 u' \eta }{(1+u'^2)^\frac{7}{2}} \dx \\ & \quad   - 5 \int_0^1 \frac{u''^2 \theta}{(1+u'^2)^\frac{7}{2}} \dx \int_0^1 \eta(r) \dr + \int_0^1 m \eta \dx - \int_0^1 m \theta \dx  \int_0^1 \eta(r) \dr   
  \\ & = \int_0^1 \eta(r)  \left[ \int_r^1 \dot{u}^{\tau}(t) \dx + \int_0^1 \dot{u}^{\tau}(t) \int_0^x \theta(s) \ds \dx - \int_0^1 \frac{u'' \theta '}{(1+ u'^2)^\frac{5}{2}} \dx \right. \\ & \left. \quad +  5 \frac{u''^2(r) u'(r)}{(1+ u'(r)^2)^\frac{7}{2}}  - 5 \int_0^1 \frac{u''^2 \theta}{(1+ u'^2)^\frac{7}{2}} \dx + m(r) - \int_0^1 m \theta \dx \right]\dr .
  \end{align}
Since $\eta \in C_0^\infty(0,1)$ was arbitrary, we infer that $\frac{u''}{(1+ u'^2)^\frac{5}{2}} \in W_{loc}^{1,1}(0,1)$ and 
\begin{align}\label{eq:dingskirchen}
\left( \frac{u''}{(1+ u'^2)^\frac{5}{2}} \right)'& = \left[ \int_r^1 \dot{u}^{\tau}(t) \dx + \int_0^1 \dot{u}^{\tau}(t) \int_0^x \theta(s) \ds \dx - \int_0^1 \frac{u'' \theta '}{(1+ u'^2)^\frac{5}{2}} \dx \right. \nonumber \\ & \left. \quad +  5 \frac{u''^2(r) u'(r)}{(1+ u'(r)^2)^\frac{7}{2}}  - 5 \int_0^1 \frac{u''^2 \theta}{(1+ u'^2)^\frac{7}{2}} \dx + m(r) - \int_0^1 m \theta \dx \right].
\end{align}
Note that the right hand side of the previous equation lies in $L^1$ and hence $\left( \frac{u''}{(1+ u'^2)^\frac{5}{2}} \right) \in W^{1,1}$. 
Using similar estimates as above \eqref{eq:dingskirchen} implies
\begin{equation}
\left\Vert \left( \frac{u''}{(1+ u'^2)^\frac{5}{2}} \right)' \right\Vert_{L^1} \leq C_1 ||\dot{u}^{\tau}(t)||_{L^2} + C_2 ( 1+ \mathcal{E}(u_0) ) + C_3 ||m||_\infty
\end{equation}
for some $C_1, C_2 ,C_3 > 0$ that can be chosen independently of $\tau, u_0$. Using \eqref{eq:utau} we find that there exist $\widetilde{C}_1$ and $\widetilde{C}_2 > 0 $ such that  
\begin{equation}
\left\Vert \left( \frac{u''}{(1+ u'^2)^\frac{5}{2}} \right)' \right\Vert_{L^1} \leq \widetilde{C}_1 ||\dot{u}||_{L^2} + \widetilde{C}_2 ( 1+ \mathcal{E}(u_0)).
\end{equation} 
From this  and the fact that $||u'||_{L^\infty(0,1)}$ is bounded independently of $\tau$ we infer that $u'' \in L^\infty(0,1)$ and 
\begin{align*}
||u''||_{L^\infty} & \leq (1 + ||u'||^2_\infty)^\frac{5}{2} \left\Vert \frac{u''}{(1+(u')^2)^\frac{5}{2}} \right\Vert_\infty 
 \leq (1 + ||u'||^2_\infty)^\frac{5}{2} \left\Vert \frac{u''}{(1+(u')^2)^\frac{5}{2}} \right\Vert_{	W^{1,1}}
\\ & \leq   (1 + ||u'||^2_\infty)^\frac{5}{2} \left(  \int_0^1 \frac{|u''|}{(1+ u'^2)^\frac{5}{2}} \dx  + \widetilde{C}_1 ||\dot{u}^{\tau}(t)||_{L^2} + \widetilde{C}_2 ( 1 + \mathcal{E}(u_0)) \right)\\ &\leq C_4 || \dot{u}^{\tau}(t) ||_{L^2}  + C_5 
\end{align*}
for some $C_4,C_5 > 0$ independent of $\tau$. With this additional information we can go back to \eqref{eq:dingskirchen} and prove that 
\begin{equation}
||u'''||_{L^\infty} \leq C_6 ||\dot{u}^\tau(t) ||_{L^2} + C_7, 
\end{equation}
whereupon the claimed estimate follows. The proof that $u''(0) = u''(1) = 0$ is very similar to \cite[Corollary 3.3]{Anna}. 
  \end{proof}

\section{Completion of the Proof of Lemma \ref{lem:symcone}}\label{app:B}
 It remains to show the strict monotonicity of $H(A)$, which is defined as in  \eqref{eq:H}. We show that $H$ is differentiable and $H' > 0$. Since we work with hypergeometric functions, we need some preliminary notation.
 \begin{definition}[{Hypergeometric Function, see \cite[Definition 2.1.5]{Askey}}]
 Let $a,b,c,z \in \mathbb{C}$. We define for $n \in \mathbb{N}$ 
 \begin{equation}
 (a)_n := \frac{\Gamma(a+n)}{\Gamma(a)} = a \cdot (a+1) \cdot ... \cdot (a+n-1) ,
 \end{equation}
 where $\Gamma$ denotes Euler's Gamma Function. We define $\HF(a,b,c, \cdot)$ to be the unique analytic continuation of 
 \begin{equation}
 B_1(0) \ni z \mapsto \sum_{n = 1}^\infty \frac{(a)_n (b)_n}{(c)_n n!} z^n \in \mathbb{C}
 \end{equation}
 \end{definition}
 We also recall the famous Pfaff Transformation (cf. \cite[Theorem 2.2.5]{Askey})
 \begin{equation}
 \HF (a,b,c,z) = \frac{1}{(1-z)^a} \HF\left(a, c-b, c , \frac{z}{z-1}\right) .
 \end{equation}

  Note that by \cite[Lemma C5]{Marius1} and the Pfaff Transformation 
 \begin{align}
 H(A) & = \frac{1}{2}  \frac{\int_{0}^{A} \frac{z}{\sqrt{A - z} {(1+z^2)^\frac{5}{4}}} \dz }{\int_{0}^{A} \frac{1}{\sqrt{A - z} {(1+z^2)^\frac{5}{4}}} \dz} = \frac{1}{3} A \frac{\HF(1, \frac{3}{2} , \frac{7}{4}, -A^2) }{\HF ( 1, \frac{1}{2}, \frac{3}{4}, -A^2) }
 \\ & = \frac{1}{3} A \frac{\HF \left(1, \frac{1}{4} , \frac{7}{4}, \frac{A^2}{1+A^2} \right) }{\HF \left( 1, \frac{1}{4}, \frac{3}{4}, \frac{A^2}{1+ A^2} \right) } = \frac{1}{3} A \frac{\sum_{k = 0}^\infty  \frac{(\nicefrac{1}{4})_k}{(\nicefrac{7}{4})_k}  \left( \frac{A^2}{1+ A^2} \right)^k}{\sum_{k = 0}^\infty  \frac{(\nicefrac{1}{4})_k}{(\nicefrac{3}{4})_k}  \left( \frac{A^2}{1+ A^2} \right)^k}
 \end{align}
 where the last step is justified as $ \frac{A^2}{1+ A^2} \in [0,1)$ for each $A \in \mathbb{R}$. 
 For the computation to come we introduce the following notation. We write $x := \frac{A^2}{1+ A^2} \in [0,1)$ and set $D(x):= \sum_{k = 0}^\infty  \frac{(\nicefrac{1}{4})_k}{(\nicefrac{3}{4})_k}  x^k $ as well as $r(x) := \HF \left(1, \frac{1}{4} , \frac{7}{4}, x \right)$. We will also use the hypergeometric equation (cf. \cite[Equation (2.3.5)]{Askey})  for $r$, which reads
 \begin{equation}
 x(1-x)r''(x) + \left( \frac{7}{4} - \frac{9}{4}x \right) r'(x) - \frac{1}{4} r(x) = 0.
 \end{equation}
 We compute the derivative and perform some rearrangements 
 \begin{align*}
 H'(A)&  =  \frac{1}{3}  \frac{\sum_{k = 0}^\infty  \frac{(\nicefrac{1}{4})_k}{(\nicefrac{7}{4})_k}  \left( \frac{A^2}{1+ A^2} \right)^k}{\sum_{k = 0}^\infty  \frac{(\nicefrac{1}{4})_k}{(\nicefrac{3}{4})_k}  \left( \frac{A^2}{1+ A^2} \right)^k}+ \frac{1}{3} A  \frac{\sum_{k = 0}^\infty k  \frac{(\nicefrac{1}{4})_k}{(\nicefrac{7}{4})_k}  \left( \frac{A^2}{1+ A^2} \right)^{k-1} }{\sum_{k = 0}^\infty  \frac{(\nicefrac{1}{4})_k}{(\nicefrac{3}{4})_k}  \left( \frac{A^2}{1+ A^2} \right)^k} \frac{2A}{(1+A^2)^2}
 \\ & - \frac{1}{3} A   \frac{\sum_{k = 0}^\infty  \frac{(\nicefrac{1}{4})_k}{(\nicefrac{7}{4})_k}  \left( \frac{A^2}{1+ A^2} \right)^k}{ \left( \sum_{k = 0}^\infty  \frac{(\nicefrac{1}{4})_k}{(\nicefrac{3}{4})_k}  \left( \frac{A^2}{1+ A^2} \right)^k\right) ^2} \sum_{k = 0}^\infty k \frac{(\nicefrac{1}{4})_k}{(\nicefrac{3}{4})_k}  \left( \frac{A^2}{1+ A^2} \right)^{k-1} \frac{2A}{(1+A^2)^2}
 \\ & = \frac{1}{3D(x)^2} \left[  \sum_{k = 0}^\infty  \frac{(\nicefrac{1}{4})_k}{(\nicefrac{7}{4})_k}  \left( \frac{A^2}{1+ A^2} \right)^k \sum_{k = 0}^\infty  \frac{(\nicefrac{1}{4})_k}{(\nicefrac{3}{4})_k}  \left( \frac{A^2}{1+ A^2} \right)^k \right. \\ & \left.\qquad \qquad \qquad + \frac{2}{1+A^2}\sum_{k = 0}^\infty k \frac{(\nicefrac{1}{4})_k}{(\nicefrac{7}{4})_k}  \left( \frac{A^2}{1+ A^2} \right)^k \sum_{k = 0}^\infty  \frac{(\nicefrac{1}{4})_k}{(\nicefrac{3}{4})_k}  \left( \frac{A^2}{1+ A^2} \right)^k \right. \\ & \left.\qquad \qquad \qquad - \frac{2}{1+ A^2}\sum_{k = 0}^\infty  \frac{(\nicefrac{1}{4})_k}{(\nicefrac{7}{4})_k}  \left( \frac{A^2}{1+ A^2} \right)^k \sum_{k = 0}^\infty k \frac{(\nicefrac{1}{4})_k}{(\nicefrac{3}{4})_k}  \left( \frac{A^2}{1+ A^2} \right)^k \right]
 \\ & = \frac{1}{3D(x)^2} \left[  \sum_{k = 0}^\infty  \frac{(\nicefrac{1}{4})_k}{(\nicefrac{7}{4})_k}  x^k \sum_{k = 0}^\infty  \frac{(\nicefrac{1}{4})_k}{(\nicefrac{3}{4})_k}  x^k + 2(1-x) \sum_{k = 0}^\infty k \frac{(\nicefrac{1}{4})_k}{(\nicefrac{7}{4})_k} x^k \sum_{k = 0}^\infty  \frac{(\nicefrac{1}{4})_k}{(\nicefrac{3}{4})_k}  x^k \right. \\ & \left. \qquad \qquad \qquad  - 2(1-x) \sum_{k = 0}^\infty  \frac{(\nicefrac{1}{4})_k}{(\nicefrac{7}{4})_k}  x^k \sum_{k = 0}^\infty k \frac{(\nicefrac{1}{4})_k}{(\nicefrac{3}{4})_k}  x^k \right]
 \\ & = \frac{1}{3D(x)^2} \left[  \sum_{k = 0}^\infty  \frac{(\nicefrac{1}{4})_k}{(\nicefrac{7}{4})_k}  x^k \sum_{k = 0}^\infty  \frac{(\nicefrac{1}{4})_k}{(\nicefrac{3}{4})_k}  x^k  \right.
  \\&  \left. \qquad \qquad \qquad  +2 (1-x) \sum_{k = 0}^{\infty} x^k\sum_{l = 0}^k l \left(  \frac{(\nicefrac{1}{4})_l (\nicefrac{1}{4})_{k-l}}{(\nicefrac{7}{4})_l(\nicefrac{3}{4})_{k-l}} - \frac{(\nicefrac{1}{4})_l(\nicefrac{1}{4})_{k-l}}{(\nicefrac{3}{4})_l(\nicefrac{7}{4})_{k-l}}\right)  \right] 
  \\ & = \frac{1}{3D(x)^2} \left[  \sum_{k = 0}^\infty  \frac{(\nicefrac{1}{4})_k}{(\nicefrac{7}{4})_k}  x^k \sum_{k = 0}^\infty  \frac{(\nicefrac{1}{4})_k}{(\nicefrac{3}{4})_k}  x^k  \right.
  \\&  \left.   +2 (1-x) \sum_{k = 0}^{\infty} x^k\sum_{l = 0}^k l (\nicefrac{1}{4})_l (\nicefrac{1}{4})_{k-l} \left(  \frac{\frac{3}{4}+ k -l}{ \frac{3}{4}(\nicefrac{7}{4})_l(\nicefrac{7}{4})_{k-l}} - \frac{\frac{3}{4}+ l}{\frac{3}{4}(\nicefrac{7}{4})_l(\nicefrac{7}{4})_{k-l}}\right)  \right]
  \\  & = \frac{1}{3D(x)^2} \left[  \sum_{k = 0}^\infty  \frac{(\nicefrac{1}{4})_k}{(\nicefrac{7}{4})_k}  x^k \sum_{k = 0}^\infty  \frac{(\nicefrac{1}{4})_k}{(\nicefrac{3}{4})_k}  x^k  \right.
  \\&  \left. \qquad \qquad \qquad   +\frac{8}{3} (1-x) \sum_{k = 0}^{\infty} x^k\sum_{l = 0}^k l(k-2l)  \frac{(\nicefrac{1}{4})_l (\nicefrac{1}{4})_{k-l}}{(\nicefrac{1}{4})_l (\nicefrac{1}{4})_{k-l}}  \right]
  \\ & = \frac{1}{3D(x)^2} \left[  \sum_{k = 0}^\infty  \frac{(\nicefrac{1}{4})_k}{(\nicefrac{7}{4})_k}  x^k \sum_{k = 0}^\infty  \frac{(\nicefrac{1}{4})_k}{(\nicefrac{7}{4})_k} \frac{\frac{3}{4}+ k}{\frac{3}{4}}  x^k  \right.
  \\&  \left. \qquad \qquad \qquad   +\frac{8}{3} (1-x) \sum_{k = 0}^{\infty} x^k\sum_{l = 0}^k l(k-2l)  \frac{(\nicefrac{1}{4})_l (\nicefrac{1}{4})_{k-l}}{(\nicefrac{7}{4})_l (\nicefrac{7}{4})_{k-l}}  \right]
   \\ & = \frac{1}{3D(x)^2} \left[ \left( \sum_{k = 0}^\infty  \frac{(\nicefrac{1}{4})_k}{(\nicefrac{7}{4})_k}  x^k \right)^2 + \frac{4}{3}\sum_{k = 0}^\infty  \frac{(\nicefrac{1}{4})_k}{(\nicefrac{7}{4})_k} k  x^k \sum_{k = 0}^\infty  \frac{(\nicefrac{1}{4})_k}{(\nicefrac{7}{4})_k}   x^k  \right.
  \\&  \left.  +\frac{8}{3} (1-x) \sum_{k = 0}^{\infty} x^k\sum_{l = 0}^k l(k-l)  \frac{(\nicefrac{1}{4})_l (\nicefrac{1}{4})_{k-l}}{(\nicefrac{7}{4})_l (\nicefrac{7}{4})_{k-l}}-\frac{8}{3} (1-x) \sum_{k = 0}^{\infty} x^k\sum_{l = 0}^k l^2  \frac{(\nicefrac{1}{4})_l (\nicefrac{1}{4})_{k-l}}{(\nicefrac{7}{4})_l (\nicefrac{7}{4})_{k-l}}   \right]  
  \\ &  = \frac{1}{3D(x)^2} \left[ \left( \sum_{k = 0}^\infty  \frac{(\nicefrac{1}{4})_k}{(\nicefrac{7}{4})_k}  x^k \right)^2 + \frac{4}{3}\sum_{k = 0}^\infty  \frac{(\nicefrac{1}{4})_k}{(\nicefrac{7}{4})_k} k  x^k \sum_{k = 0}^\infty  \frac{(\nicefrac{1}{4})_k}{(\nicefrac{7}{4})_k}   x^k  \right.
  \\&  \left.  +\frac{8}{3} (1-x) \left( \sum_{k = 0}^{\infty} k\frac{(\nicefrac{1}{4})_k}{(\nicefrac{7}{4})_k} x^k \right)^2-\frac{8}{3} (1-x) \sum_{k = 0}^{\infty} k^2\frac{(\nicefrac{1}{4})_k}{(\nicefrac{7}{4})_k}x^k  \sum_{k = 0}^{\infty} \frac{(\nicefrac{1}{4})_k}{(\nicefrac{7}{4})_k}x^k   \right]  
  \\ &  = \frac{1}{3D(x)^2} \left[ \left( \sum_{k = 0}^\infty  \frac{(\nicefrac{1}{4})_k}{(\nicefrac{7}{4})_k}  x^k \right)^2 + \frac{4}{3}\sum_{k = 0}^\infty  \frac{(\nicefrac{1}{4})_k}{(\nicefrac{7}{4})_k} k  x^k \sum_{k = 0}^\infty  \frac{(\nicefrac{1}{4})_k}{(\nicefrac{7}{4})_k}   x^k  \right.
  \\&  \left.  +\frac{8}{3} (1-x) \left( \sum_{k = 0}^{\infty} k\frac{(\nicefrac{1}{4})_k}{(\nicefrac{7}{4})_k} x^k \right)^2-\frac{8}{3} (1-x) \sum_{k = 0}^{\infty} [k(k-1) + k]\frac{(\nicefrac{1}{4})_k}{(\nicefrac{7}{4})_k}x^k  \sum_{k = 0}^{\infty} \frac{(\nicefrac{1}{4})_k}{(\nicefrac{7}{4})_k}x^k   \right]  
  \\ & = \frac{1}{3D(x)^2} \left( r(x)^2 + \frac{4}{3}x r'(x) r(x) + \frac{8}{3}(1-x)x^2 r'(x)^2  \right. \\ & \left. \qquad \qquad \qquad -\frac{8}{3}(1-x) x^2 r''(x) r(x) - \frac{8}{3}(1-x) xr'(x) r(x) \right) 
  \\ &= \frac{1}{3D(x)^2} \left( r(x)^2 + \frac{4}{3}x r'(x) r(x) + \frac{8}{3}(1-x)x^2 r'(x)^2 + \frac{8}{3}x \left( \frac{7}{4}- \frac{9}{4}x \right) r'(x) r(x) \right. \\ & \left. \qquad \qquad \qquad - \frac{2}{3}x r(x)^2 - \frac{8}{3}(1-x) xr'(x) r(x) \right)
   \\ &=  \frac{1}{3D(x)^2} \left( r(x)^2 - \frac{2}{3}x r(x)^2 + \frac{8}{3}(1-x)x^2 r'(x)^2 + \frac{10}{3}x r(x) r'(x) - \frac{10}{3}x^2 r(x) r'(x)  \right).
 \end{align*}
As $r(x), r'(x)$ are power series with only positive coefficients, they are themselves positive on $(0,1)$. As $x \in (0,1)$ we can estimate $x^2 \leq x$ and $x\leq 1$ to obtain  
\begin{equation}
H'(A) \geq \frac{1}{3D(x)^2} \left( \frac{1}{3}r(x)^2 + \frac{8}{3} (1-x) x^2r'(x)^2 \right) > 0 . 
\end{equation}
The claim follows. 

\section{Proof of Lemma \ref{lem:75}}\label{ref:appc}
\begin{proof}
Let $f$ be as in the statement. First note that 
\begin{equation}
f'(\nicefrac{1}{2}) = J(f(\nicefrac{1}{2}) ) = J(x_0) = 0.
\end{equation}
As $f$ is strictly concave, $f'$ is strictly decreasing and hence $f' > 0 $ on $(0, \frac{1}{2})$. This implies in particular that $0 <f(r) < x_0$ for all $r \in (0, \frac{1}{2})$. In particular $J(f(r)) > 0 $ for all $r \in (0,\frac{1}{2})$. Hence we may write 
\begin{equation}
\frac{f'(r)}{J(f(r))} = 1 \quad \forall r \in (0, \frac{1}{2}) .
\end{equation}
Now fix $s \in (0,\frac{1}{2})$ and choose $r_n \rightarrow \frac{1}{2}$ a monotone sequence such that $r_n < \frac{1}{2}$ for all $n \in \mathbb{N}$. Integrate from $s$ to $r_n$ to find 
\begin{equation}
\int_s^{r_n} \frac{f'(r)}{J(f(r))} \dr = r_n - s. 
\end{equation}
As $J$ is locally Lipschitz on $(0,x_0)$ we can use the substitution rule to get 
\begin{equation}
\int_{f(s)}^{f(r_n)}  \frac{1}{J(z)} \dz = r_n - s.
\end{equation}
Note that $\frac{1}{J} >0 $ in the domain of integration. As $f' >0 $ on $(0, \frac{1}{2})$ we find that $f$ is monotone and hence $f(r_n)$ converges monotonically to $f(\frac{1}{2}) = x_0$. We can apply the monotone convergence theorem to pass to the limit and find 
\begin{equation}
\int_{f(s)}^{x_0} \frac{1}{J(z)} \dz = \frac{1}{2} - s \quad  \forall s \in (0, \nicefrac{1}{2}).
\end{equation}
As the integrand is positive on $(0,x_0)$, the integral is strictly monotone in its lower argument and as a result $f(s)$ is uniquely determined for each $s \in (0,\frac{1}{2})$. As we have required that $f \in C^2( [0, \frac{1}{2} ] )$, $f$ is also uniquely determined at the boundary points.
\end{proof}


\begin{thebibliography}{10}
\bib{Askey}{book}{
   author={Andrews, George E.},
   author={Askey, Richard},
   author={Roy, Ranjan},
   title={Special functions},
   series={Encyclopedia of Mathematics and its Applications},
   volume={71},
   publisher={Cambridge University Press, Cambridge},
   date={1999},
   pages={xvi+664},
   isbn={0-521-62321-9},
   isbn={0-521-78988-5},
   review={\MR{1688958}},
   doi={10.1017/CBO9781107325937},
}
\bib{usersguide}{article}{
   author={Ambrosio, Luigi},
   author={Gigli, Nicola},
   title={A user's guide to optimal transport},
   conference={
      title={Modelling and optimisation of flows on networks},
   },
   book={
      series={Lecture Notes in Math.},
      volume={2062},
      publisher={Springer, Heidelberg},
   },
   date={2013},
   pages={1--155},
   review={\MR{3050280}},
   doi={10.1007/978-3-642-32160-3},
}
\bib{Ambrosio}{book}{
   author={Ambrosio, Luigi},
   author={Gigli, Nicola},
   author={Savar\'{e}, Giuseppe},
   title={Gradient flows in metric spaces and in the space of probability
   measures},
   series={Lectures in Mathematics ETH Z\"{u}rich},
   edition={2},
   publisher={Birkh\"{a}user Verlag, Basel},
   date={2008},
   pages={x+334},
   isbn={978-3-7643-8721-1},
   review={\MR{2401600}},
}
\bib{Caffarelli}{article}{
   author={Athanasopoulos, Ioannis},
   author={Caffarelli, Luis},
   author={Milakis, Emmanouil},
   title={Parabolic obstacle problems, quasi-convexity and regularity},
   journal={Ann. Sc. Norm. Super. Pisa Cl. Sci. (5)},
   volume={19},
   date={2019},
   number={2},
   pages={781--825},
   issn={0391-173X},
   review={\MR{3964414}},
}
\bib{Blatt}{article}{
author={Blatt, Simon},
author={Hopper, Christopher},
author={Vorderobermeier, Nicole},
title={A minimising movement scheme for the p-elastic energy of curves},
journal={Preprint},
date={2021},
eprint={https://arxiv.org/abs/2101.10101}
}
\bib{Chae}{book}{
   author={Chae, Soo Bong},
   title={Lebesgue integration},
   series={Universitext},
   edition={2},
   publisher={Springer-Verlag, New York},
   date={1995},
   pages={xiv+264},
   isbn={0-387-94357-9},
   review={\MR{1369771}},
   doi={10.1007/978-1-4612-0781-8},
}
\bib{Anna}{article}{
   author={Dall'Acqua, Anna},
   author={Deckelnick, Klaus},
   title={An obstacle problem for elastic graphs},
   journal={SIAM J. Math. Anal.},
   volume={50},
   date={2018},
   number={1},
   pages={119--137},
   issn={0036-1410},
   review={\MR{3742685}},
   doi={10.1137/17M111701X},
}
\bib{Dayrens}{article}{
   author={Dayrens, Fran\c{c}ois},
   author={Masnou, Simon},
   author={Novaga, Matteo},
   title={Existence, regularity and structure of confined elasticae},
   journal={ESAIM Control Optim. Calc. Var.},
   volume={24},
   date={2018},
   number={1},
   pages={25--43},
   issn={1292-8119},
   review={\MR{3764132}},
   doi={10.1051/cocv/2016073},
}
\bib{Dayrens2}{article}{
   author={Dayrens, Fran\c{c}ois},
   title={Minimisations sous contraintes et flots du périmètre et de l’énergie de Willmore},
   journal={PhD thesis, Université de Lyon,},
    language={French},
   date={2016},
   eprint= {https://tel.archives-ouvertes.fr/tel-01400613/document},
}
\bib{Dreher}{article}{
   author={Dreher, Michael},
   author={J\"{u}ngel, Ansgar},
   title={Compact families of piecewise constant functions in $L^p(0,T;B)$},
   journal={Nonlinear Anal.},
   volume={75},
   date={2012},
   number={6},
   pages={3072--3077},
   issn={0362-546X},
   review={\MR{2890969}},
   doi={10.1016/j.na.2011.12.004},
}
\bib{Euler}{article}{
   author={Euler, Leonhard},
   title={Methodus
   inveniendi lineas curvas maximi minimive proprietate gaudentes sive
   solutio problematis isoperimetrici latissimo sensu accepti},
   language={Latin},
   date={1744},
   eprint={http://eulerarchive.maa.org/},
}
\bib{Grossmann}{book}{
   author={Grossmann, Christian},
   author={Roos, Hans-G\"{o}rg},
   title={Numerical treatment of partial differential equations},
   series={Universitext},
   note={Translated and revised from the 3rd (2005) German edition by Martin
   Stynes},
   publisher={Springer, Berlin},
   date={2007},
   pages={xii+591},
   isbn={978-3-540-71582-5},
   review={\MR{2362757}},
   doi={10.1007/978-3-540-71584-9},
}
\bib{Sweers}{book}{
   author={Gazzola, Filippo},
   author={Grunau, Hans-Christoph},
   author={Sweers, Guido},
   title={Polyharmonic boundary value problems},
   series={Lecture Notes in Mathematics},
   volume={1991},
   note={Positivity preserving and nonlinear higher order elliptic equations
   in bounded domains},
   publisher={Springer-Verlag, Berlin},
   date={2010},
   pages={xviii+423},
   isbn={978-3-642-12244-6},
   review={\MR{2667016}},
   doi={10.1007/978-3-642-12245-3},
}
\bib{LiebLoss}{book}{
   author={Lieb, Elliott H.},
   author={Loss, Michael},
   title={Analysis},
   series={Graduate Studies in Mathematics},
   volume={14},
   edition={2},
   publisher={American Mathematical Society, Providence, RI},
   date={2001},
   pages={xxii+346},
   isbn={0-8218-2783-9},
   review={\MR{1817225}},
   doi={10.1090/gsm/014},
}
\bib{Miura1}{article}{
   author={Miura, Tatsuya},
   title={Overhanging of membranes and filaments adhering to periodic graph
   substrates},
   journal={Phys. D},
   volume={355},
   date={2017},
   pages={34--44},
   issn={0167-2789},
   review={\MR{3683120}},
   doi={10.1016/j.physd.2017.06.002},
}
\bib{Miura2}{article}{
   author={Miura, Tatsuya},
   title={Singular perturbation by bending for an adhesive obstacle problem},
   journal={Calc. Var. Partial Differential Equations},
   volume={55},
   date={2016},
   number={1},
   pages={Art. 19, 24},
   issn={0944-2669},
   review={\MR{3456944}},
   doi={10.1007/s00526-015-0941-z},
}
\bib{Miura3}{article}{
   author={Miura, Tatsuya},
   title={Polar tangential angles and free elasticae},
   journal={Preprint},
   date={2020},
   url={https://arxiv.org/pdf/2004.06497.pdf}
}
\bib{Marius1}{article}{
   author={M\"{u}ller, Marius},
   title={An obstacle problem for elastic curves: existence results},
   journal={Interfaces Free Bound.},
   volume={21},
   date={2019},
   number={1},
   pages={87--129},
   issn={1463-9963},
   review={\MR{3951579}},
   doi={10.4171/IFB/418},
}
\bib{Marius2}{article}{
   author={M\"{u}ller, Marius},
   title={On gradient flows with obstacles and Euler's elastica},
   journal={Nonlinear Anal.},
   volume={192},
   date={2020},
   pages={111676, 48},
   issn={0362-546X},
   review={\MR{4031155}},
   doi={10.1016/j.na.2019.111676},
}
\bib{Novaga1}{article}{
   author={Novaga, Matteo},
   author={Okabe, Shinya},
   title={The two-obstacle problem for the parabolic biharmonic equation},
   journal={Nonlinear Anal.},
   volume={136},
   date={2016},
   pages={215--233},
   issn={0362-546X},
   review={\MR{3474411}},
   doi={10.1016/j.na.2016.02.004},
}
\bib{Novaga2}{article}{
   author={Novaga, Matteo},
   author={Okabe, Shinya},
   title={Regularity of the obstacle problem for the parabolic biharmonic
   equation},
   journal={Math. Ann.},
   volume={363},
   date={2015},
   number={3-4},
   pages={1147--1186},
   issn={0025-5831},
   review={\MR{3412355}},
   doi={10.1007/s00208-015-1200-5},
}

\bib{Yoshizawa}{article}{
   author={Okabe, Shinya},
   author={Yoshizawa, Kensuke},
   title={A dynamical approach to the variational inequality on modified
   elastic graphs},
   journal={Geom. Flows},
   volume={5},
   date={2020},
   number={1},
   pages={78--101},
   review={\MR{4171115}},
   doi={10.1515/geofl-2020-0100},
}
\bib{Talenti}{article}{
   author={Talenti, Giorgio},
   title={Nonlinear elliptic equations, rearrangements of functions and
   Orlicz spaces},
   journal={Ann. Mat. Pura Appl. (4)},
   volume={120},
   date={1979},
   pages={160--184},
   issn={0003-4622},
   review={\MR{551065}},
   doi={10.1007/BF02411942},
}
\bib{Tartar}{book}{
   author={Tartar, Luc},
   title={An introduction to Sobolev spaces and interpolation spaces},
   series={Lecture Notes of the Unione Matematica Italiana},
   volume={3},
   publisher={Springer, Berlin; UMI, Bologna},
   date={2007},
   pages={xxvi+218},
   isbn={978-3-540-71482-8},
   isbn={3-540-71482-0},
   review={\MR{2328004}},
}
\bib{Yoshizawa2}{article}{
author={Yoshizawa, Kensuke},
title={A remark on elastic graphs with the symmetric cone obstacle},
journal={To appear in SIAM Journal of Mathematical Analysis},
date={2021}

}


\end{thebibliography}
\end{document}